\documentclass[12pt]{article}

\usepackage[]{graphicx,float,latexsym,times}
\usepackage{amsfonts,amstext,amsmath,amssymb,amsthm}
\usepackage{caption2}
\usepackage{multirow}
\usepackage{subfigure}
\usepackage{setspace}

  \newtheorem{innercustomlem}{Lemma}
\newenvironment{customlem}[1]
  {\renewcommand\theinnercustomlem{#1}\innercustomlem}
  {\endinnercustomlem}
  
\usepackage{xcolor}

\long\def\symbolfootnote[#1]#2{\begingroup
\def\thefootnote{\fnsymbol{footnote}}\footnote[#1]{#2}\endgroup}

\DeclareMathOperator*{\supess}{ess\,sup}

\usepackage{url} 
\usepackage{bbm}

\usepackage{titlesec} 

\titleformat{\section}{\large\bfseries}{\thesection.}{.5em}{}
\titlespacing*{\section}{0pt}{*3}{*2}
\titleformat{\subsection}{\normalfont\bfseries}{\thesubsection.}{.5em}{}
\titlespacing*{\subsection} {0pt}{*3}{*2}
\titleformat{\subsubsection}{\normalfont\bfseries}{\thesubsubsection.}{.5em}{}
\titlespacing*{\subsubsection} {0pt}{*3}{*2}

\theoremstyle{plain} 
\newtheorem{theorem}{Theorem}[section]
\newtheorem{lemma}{Lemma}[section]

\theoremstyle{definition} 

\newtheorem{remark}{Remark}[section]

\numberwithin{equation}{section} 

\usepackage[hidelinks]{hyperref}
\usepackage[authoryear]{natbib}
\usepackage{authblk}

\begin{document}

\title{Sequential Subspace Change-Point Detection}

\author[1]{Liyan Xie\thanks{lxie49@gatech.edu}}
\author[1]{Yao Xie\thanks{yao.xie@isye.gatech.edu}}
\author[2,3]{George V. Moustakides\thanks{moustaki@upatras.gr}}
\affil[1]{School of Industrial and Systems Engineering, Georgia Institute of Technology}
\affil[2]{Department of Computer Science, Rutgers University}
\affil[3]{Department of Electrical and Computer Engineering, University of Patras}

\date{}

\maketitle

%
%
%

\begin{abstract}
We consider the online monitoring of multivariate streaming data for changes that are characterized by an unknown subspace structure manifested in the covariance matrix. In particular, we consider the covariance structure changes from an identity matrix to an unknown spiked covariance model. We assume the post-change distribution is unknown, and propose two detection procedures: the Largest-Eigenvalue Shewhart chart and the Subspace-CUSUM detection procedure. We present theoretical approximations to the average run length (ARL) and the expected detection delay (EDD) for the Largest-Eigenvalue Shewhart chart and also provide analysis for tuning parameters of the Subspace-CUSUM procedure. The performance of the proposed methods is illustrated using simulation and real data for human gesture detection and seismic event detection.
\end{abstract}


\section{INTRODUCTION}\label{sec:intro}

Detecting the change from high-dimensional streaming data is a fundamental problem in various applications such as video surveillance \citep{sultani2018real}, sensor networks \citep{xie2013sequential}, wearable sensors \citep{sprint2016unsupervised}, and seismic events detection \citep{li2018high}. In many scenarios, the change happens to the covariance structure and can be represented as a low-rank subspace to capture the similarity of signal waveforms observed at multiple sensors. 
We consider the fundamental problem of detecting such a change in the covariance matrix that shifts from an identity matrix to a spiked covariance model \citep{johnstone2001distribution}. Different from the offline hypothesis test considered in \cite{berthet2013optimal}, we assume a sequential setting, where the goal is to detect such a change as quickly as possible after it occurs.

A formal description of the problem is as follows. Assume a sequence of multivariate observations $x_1, x_2, \ldots,x_t,\ldots$, where $x_t \in \mathbb{R}^k$ and $k$ is the data dimension. At a certain time $\tau$, the distribution of the observation changes. 
In particular, we are interested in structural changes that happen to the covariance matrix, which we describe below: (1) the {\it emerging subspace}, meaning the change is a subspace emerging from a noisy background and thus the covariance matrix changes from an identity matrix to a spiked covariance matrix; (2) the {\it switching subspace}, meaning that the signals are along with different subspaces before and after the change, resulting the covariance matrix to change from one spiked covariance matrix to another. The emerging subspace problem can arise, for instance, from weak signal detection for seismic sensor arrays \citep{sprint2016unsupervised}, and the switching subspace detection can arise from monitoring principal component analysis (PCA) for streaming data \citep{balzano2018streaming}. The switching subspace problem, as we will show, can be reduced to the emerging subspace problem; therefore, we focus on the emerging subspace problem. 

The main contribution of this paper is two-fold. 
From the methodology perspective, we propose two sequential detection procedures: the Largest-Eigenvalue Shewhart chart and the Subspace-CUSUM procedure. The Largest-Eigenvalue Shewhart chart computes the largest eigenvalue of the sample covariance matrix over a sliding window and detects a change when the statistic exceeds the threshold. 
The Subspace-CUSUM is derived based on the likelihood ratio following the approach of classical CUSUM \citep{page1954continuous}, but instead of assuming the parameters are fully specified, we estimate the parameters and plug-in, which is analogous to the generalized likelihood ratio (GLR) statistic \citep{lai1995sequential}.  
From the theoretical perspective, we provide a theoretical analysis of the proposed procedures, which facilitates efficient calibration of the parameters. We consider two commonly used metrics: the {\it average run length} (ARL) and the {\it expected detection delay} (EDD). Theoretical approximations can help us determine the threshold in the detection procedure efficiently. Moreover, building on Anderson's results for the distribution of eigenvectors \citep{anderson1963asymptotic}, we provide theoretical guidelines on how to choose the parameters involved in the Subspace-CUSUM procedure. 

The proposed detection procedures are computationally efficient since they only require computing the leading eigenvalue and eigenvector of the sample covariance matrix, respectively. They are widely applicable to real data whenever there is a low-rank subspace change. For example, we have demonstrated its use in human activity detection from wearable sensors data and seismic event detection. 


\subsection{Related Work} \label{sec:literature}  
In change-point detection and industrial quality control, commonly used methods can be categorized into Shewhart chart, CUSUM, and generalized likelihood ratio (GLR) types of detection procedures.

Shewhart charts can be viewed as scan-statistics over time. A change is detected when the process is out-of-control, i.e., the detection statistic falls out of the control limit.  
A commonly used Shewhart chart for multivariate observations is the Hotelling $T^2$ control chart \citep{hotelling1947multivariate}, which can detect both mean and covariance shifts and the control limits are set through chi-square distributions. Modified $T^2$ charts based on principal component analysis are considered in \cite{jackson1959quality,jackson1979control}. The $U^2$ multivariate control chart in \cite{runger1996projections} considers detecting the mean shift in a known subspace. Those work does not consider the largest eigenvalue as a detection statistic. 

While Shewhart charts use the current subgroup samples to compute the detection statistic, the CUSUM procedure utilizes all past samples and updates the detection statistic recursively based on the log-likelihood ratio \citep{page1954continuous}. 
Multivariate CUSUM procedure for detecting mean shift has been developed in \cite{pignatiello1990comparisons} and a more recent work \citep{bodnar2005multivariate} presents CUSUM based on projected data. In classic CUSUM, the pre-change and post-change distributions are specified completely. The Subspace-CUSUM procedure here is not a typical CUSUM since we estimate the unknown subspace after the change. 

Usually, the post-change distributions or their parameters are unknown and hard to pre-specify. One solution is to set the post-change parameter to represent the ``smallest possible change'' of interest. However, when there is a parameter mismatch, the CUSUM procedure suffers from a performance loss. The generalized likelihood ratio (GLR) procedure is introduced to handle unknown post-change distributions \citep{lai1995sequential}. 
The Subspace-CUSUM procedure here is different from the GLR procedure since we do not estimate the full log-likelihood function; instead, we only estimate the subspace and introduce an additional parameter to control the performance. 

Covariance shift detection has been considered in the past literature using various detection statistics. A multivariate CUSUM based on likelihood functions of multivariate Gaussian is studied in \cite{healy1987note} considering a specific setting where the covariance changes from $\Sigma$ to $c\Sigma$ for a constant $c$. The determinant of the sample covariance matrix was used in \cite{alt1985multivariate} and \cite{alt1988multivariate} to detect the covariance change. \cite{chan2001cumulative} considers a CUSUM chart for monitoring covariance shift using the projection pursuit \citep{huber1985projection} and likelihood ratio, with simulation studies on the performance of the proposed methods. Offline change detection of covariance change is studied in \cite{chen2004statistical} using the Schwarz information criterion \citep{schwarz1978estimating}, where the change-point location must satisfy certain regularity condition to ensure the existence of the maximum likelihood estimator. Recently, \cite{wang2017optimal} uses the wide binary segmentation through independent projection (WBSIP) to recover the change-points for the covariance matrix in the offline setting. \cite{avanesov2018change} uses the distance between empirical precision matrices to detect abrupt changes in covariance for the offline case. Classical approaches usually consider the general setting, and here we are interested in detecting the structural change, i.e., spiked covariance matrix.  

Recent work has also considered other types of structured covariance changes. The detection of a shift in an off-diagonal sub-matrix of the covariance matrix is studied in \cite{arias2012detection} using likelihood ratios. The detection of switching subspaces is studied in  \cite{jiao2018subspace} based on a CUSUM type procedure, but they only estimate the pre-change subspace using historical data and assume the post-change subspace is known, this is different from our work since we also estimate the post-change subspace. \cite{zhang2018dynamic} develops an offline modeling framework for multivariate functional data based on sparse subspace clustering. 
 
The Largest-Eigenvalue Shewhart Chart is related to \cite{berthet2013optimal}, which studies the sparse principal component test based on sparse eigenvalue statistics. The largest eigenvalue statistic is shown to be asymptotically minimax optimal in \cite{berthet2013optimal} for detecting whether there exists a sparse and low-rank component. A natural sequential version of this idea is to use a sliding window and estimate the largest eigenvalue of the corresponding sample covariance matrix. However, this sequential version does not enjoy any form of (asymptotic) optimality.  

\subsection{Organization}
The rest of the paper is organized as follows. 
Section\,\ref{sec:formulation} presents the formulation of the emerging and switching subspace problems and show a unified framework. Section\,\ref{sec:detect} presents the proposed two sequential change detection procedures: the Largest-Eigenvalue Shewhart chart and Subspace-CUSUM procedure. Section\,\ref{sec:theo} presents theoretical approximations and bounds for the average run length and the expected detection delay of the Largest-Eigenvalue Shewhart chart, as well as theoretical calibration and parameter choice for Subspace-CUSUM procedure. Section\,\ref{sec:numerical} contains simulation studies to demonstrate the performance of the proposed algorithms in different settings. Section\,\ref{sec:data} shows two real data examples using human gesture detection and seismic event detection. Finally, Section\,\ref{sec:conclusion} contains our concluding remarks.

\section{PROBLEM SETUP} \label{sec:formulation} 
We first introduce the spiked covariance model considered in \cite{johnstone2001distribution}, which assumes that a small number of directions explain most of the variance. For simplicity, we consider the spiked covariance model of rank-one in this paper. The results can be generalized to the case where rank is greater than one using similar ideas. 
In particular, the rank-one spiked covariance matrix is given by 
\[
\Sigma = \sigma^2I_k + \theta uu^\intercal,
\]
where $I_k$ denotes an identity matrix of size $k$;  $\theta$ is the signal strength; $u \in \mathbb{R}^k$ represents a basis for the subspace with unit norm $\left\Vert u \right\Vert=1$; $\sigma^2$ is the noise variance, which will be considered known since it can be estimated from training data. It is possible to consider $\sigma^2$ unknown as well and provide estimates of this parameter along with the necessary estimates of $u$. However, to simplify our presentation, we decide to consider $\sigma^2$ known. The Signal-to-Noise Ratio (SNR) is defined as 
$\rho = \theta/\sigma^2$. 

Formally, the \textit{emerging} subspace problem can be cast as follows
\begin{equation}
\begin{array}{ll}
x_t  \stackrel{\text{iid}}{\sim} \mathcal{N}(0, \sigma^2I_k), &t = 1,2,\ldots,\tau, \\
x_t \stackrel{\text{iid}}{\sim} \mathcal{N}(0,  \sigma^2I_k+ \theta uu^\intercal), &t = \tau+1,\tau+2,\ldots
\end{array} 
\label{eq:hypothesis}
\end{equation}
where $\tau$ is the unknown change-point that we would like to detect from data that are acquired sequentially. 
Similarly, the \textit{switching} subspace problem can be formulated as follows
\begin{equation}
\begin{array}{ll}
x_t  \stackrel{\text{iid}}{\sim} \mathcal{N}(0, \sigma^2I_k + \theta u_1u_1^\intercal), &\!\!\!t = 1,2,\ldots,\tau, \\
x_t \stackrel{\text{iid}}{\sim} \mathcal{N}(0,  \sigma^2I_k  + \theta u_2 u_2^\intercal),&\!\!\!t = \tau+1,\tau+2,\ldots
\end{array}
\label{eq:hypothesis_switch}
\end{equation}
where $u_1, u_2 \in \mathbb R^{k}$ represent bases for the subspaces before and after the change, with $\|u_1\| = \|u_2\| = 1$ and $u_1$ is considered known. In both settings, our goal is to detect the change as quickly as possible, subject to the constraint that false detections occurring before the true change-point are very rare. 

The switching subspace problem \eqref{eq:hypothesis_switch} can be reduced into the emerging subspace problem \eqref{eq:hypothesis} by a simple data projection. In detail, we can select any orthonormal matrix $Q \in \mathbb{R}^{(k-1)\times k}$ such that
\[ Q u_1 = 0, \quad QQ^\intercal = I_{k-1},\]
which means that all rows of $Q$ are orthogonal to $u_1$, and they are orthogonal to each other and have unit norm. Then, we project each observation $x_t$ using the projection matrix $Q$ onto a $k-1$ dimensional space and obtain a new sequence 
\[ y_t = Qx_t \in \mathbb{R}^{k-1}, t = 1,2,\ldots . \]
Then $y_t$ is a zero-mean random vector with covariance matrix $\sigma^2I_{k-1}$ before the change and $\sigma^2I_{k-1} + \theta Qu_2u_2^\intercal Q^\intercal$ after the change. 
Let $u = Qu_2/\left\Vert Qu_2 \right\Vert$, and 
\[ \tilde{\theta} = \theta \left\Vert Qu_2 \right\Vert^2 = \theta [ 1- (u_1^\intercal u_2)^2]. \]
Thus, problem \eqref{eq:hypothesis_switch} can be reduced to the following
\begin{equation}
\begin{array}{ll}
y_t  \stackrel{\text{iid}}{\sim} \mathcal{N}(0, \sigma^2I_{k-1}), &\!\!\!t = 1,2,\ldots,\tau, \\
y_t \stackrel{\text{iid}}{\sim} \mathcal{N}(0,  \sigma^2I_{k-1}  +\tilde{ \theta} u u^\intercal),&\!\!\!t = \tau+1,\tau+2,\ldots
\end{array}
\label{eq:hypothesis_switch_transform}
\end{equation}
Note that this way the switching subspace problem is reduced into the emerging subspace problem, where the new signal power $\tilde \theta$ depends on the angle between $u_1$ and $u_2$, which is consistent with our intuition. 

We would like to emphasize that by projecting the observations onto a lower-dimensional space, we lose information, suggesting that the two versions of the problem are not exactly equivalent. Indeed, the optimum detector for the transformed data in \eqref{eq:hypothesis_switch_transform} and the one for the original data in \eqref{eq:hypothesis_switch} do not coincide. This can be easily verified by computing the corresponding CUSUM tests and their optimum performance. Despite this difference, it is clear that with the proposed approach, we put both problems under the same framework, offering computationally simple methods to solve the original problem in \eqref{eq:hypothesis_switch}. Consequently, in the following analysis, we focus solely on problem \eqref{eq:hypothesis}.

\section{DETECTION PROCEDURES} \label{sec:detect}

We propose two online methods: the Largest-Eigenvalue Shewhart chart and the Subspace-CUSUM procedure. 
Below, denote by $\mathbb{P}_\tau$ and $\mathbb{E}_\tau$ the probability and expectation induced when there is a change-point at the \textit{deterministic} time $\tau$. Under this definition, $\mathbb{P}_\infty$ and $\mathbb{E}_\infty$ is the probability and the expectation under the nominal regime (change never happens) while $\mathbb{P}_0$ and $\mathbb{E}_0$ the probability and expectation under the alternative regime (change happens before we take any data).

\subsection{Largest-Eigenvalue Shewhart Chart} \label{detect:eigen}

Motivated by the test statistic in \cite{berthet2013optimal}, we consider a Shewhart chart by computing the largest eigenvalue of the sample covariance matrix repeatedly over a sliding window. 
Assume the window length is $w$. For each time $t>0$, the \textit{un-normalized} sample covariance matrix using the available samples is given by 
\begin{equation}
\hat{\Sigma}_{t,\min\{t,w\}}=\left\{
\begin{array}{l}
x_1 x_1^\intercal+\cdots+x_t x_t^\intercal,~\text{for}~t<w\\
x_{t-w+1} x_{t-w+1}^\intercal+\cdots+x_t x_t^\intercal,~\text{for}~t\geq w.
\end{array}\right.
\label{sampleCov}
\end{equation}
We note that for $t=1$ the matrix contains a single outer product and as time progresses the number of outer products increases linearly until it reaches $w$. After this point, namely for $t\geq w$, the number of outer products remains equal to $w$.

Let $\lambda_{\max}(X)$ denote the largest eigenvalue of a symmetric matrix $X$. We define the {\it Largest-Eigenvalue Shewhart chart}, as the one that stops according to the following rule:
\begin{equation}
\label{eigenvalue_procedure}
T_{\rm E} = \inf\{t>0: \lambda_{\max}(\hat{\Sigma}_{t,\min\{t,w\}}) \geq b \}, 
\end{equation}
where $b > 0$ is a constant threshold selected to meet a suitable false alarm constraint.
We need to emphasize that we do {\it not} divide by $\min\{t,w\}$ when forming the un-normalized sample covariance matrix. As we explain in Section\,\ref{sec:worstedd}, it is better for $T_{\rm E}$ to always divide by $w$ instead of $\min\{t,w\}$. Consequently, we can omit the normalization constant $w$ from our detection statistics by absorbing it into the threshold.

\subsection{Subspace-CUSUM}\label{detect:exactcusum}

The CUSUM procedure \citep{page1954continuous,siegmund2013sequential} is the most popular sequential test for change detection. When the observations are independent and identically distributed before and after the change, CUSUM is known to be exactly optimum \citep{moustakides1986optimal} in the sense that it solves a very well defined constrained optimization problem introduced in \cite{lorden1971procedures}. However, the CUSUM procedure can only be applied when we have exact knowledge of the pre- and post-change distributions. For our problem, this requires complete specification of all parameters, namely the subspace $u$, noise power $\sigma^2$, and SNR $\rho$. In this section, we first introduce the exact CUSUM formulation and then present our Subspace-CUSUM procedure.

\subsubsection{Exact CUSUM}

To derive the CUSUM procedure, let $f_\infty(\cdot),f_0(\cdot)$ denote the pre- and post-change probability density function (PDF) of the observations. Then the CUSUM statistics is defined by maximizing the log-likelihood ratio statistic over all possible change-point locations
\begin{equation*}
S_ t =  \max_{1 \leq j  \leq t } \sum_{i=j}^t \log \frac{f_0(x_i)}{f_{\infty}(x_i)},
\label{cusum_exact}
\end{equation*}
which has the recursive implementation
\begin{equation}
S_t = (S_{t-1})^+ + \log \frac{f_0(x_t)}{f_{\infty}(x_t)},
\label{cusum_recursive}
\end{equation}
that enables its efficient calculation \citep{moustakides1986optimal}, where $x^{+} = \max\{0,x\}$. The CUSUM stopping time in turn is defined as
\begin{equation}
T_{\rm C} = \inf\{t>0: S_t \geq b\},
\label{cusum_procedure}
\end{equation}
where $b$ is a threshold selected to meet a suitable false alarm constraint.
For our problem of interest \eqref{eq:hypothesis} we can derive that
\begin{equation*}\label{expansion}
\begin{aligned}
\log \frac{f_0(x_t)}{f_\infty(x_t)}  & = \log\bigg[ \frac{[(2\pi)^{k}|\sigma^2I_k + \theta u u^\intercal|]^{-1/2}}{[(2\pi)^{k}\sigma^{2k}]^{-1/2}} \times
 \frac{\exp\{ - x_t^\intercal(\sigma^2I_k+\theta u u^\intercal)^{-1}x_t / 2  \} }{\exp\{- x_t^\intercal x_t /(2\sigma^2) \} } \bigg] \\
& = \log\left[  \left|I_k + \rho u u^\intercal \right|^{-\frac 1 2} 
\exp\left\{ \frac 1 2 \frac{\theta}{\theta+\sigma^2} \frac{(u^\intercal x_t)^2}{\sigma^2}  \right\} \right]  \\
&  =\frac{\rho}{2\sigma^2(1+\rho)} \Big\{(u^\intercal x_t)^2- \sigma^2 \left(1+ \frac1{\rho}\right)\log(1+\rho)\Big\}.
\end{aligned}
\end{equation*}
The second equality is due to the matrix inversion lemma \citep{woodbury1950inverting} that allows us to write
\[
(\sigma^2I_k+\theta u u^\intercal)^{-1} = 
 \frac1{\sigma^2}I_k - \frac{\theta}{\theta+\sigma^2}\frac{uu^\intercal}{\sigma^2},
\]
which, after substitution into the equation, yields the desired result. Note that the multiplicative factor $\rho/[2\sigma^2(1+\rho)]$ is positive, so we can omit it from the log-likelihood ratio when forming the CUSUM statistic in \eqref{cusum_recursive}. This leads to
\begin{equation}
S_t = (S_{t-1})^+ + (u^\intercal x_t)^2 - \sigma^2\left(1+\frac{1}{\rho}\right)\log(1+\rho).
\label{cusum_recur}
\end{equation}

\begin{remark}\label{drift_property_1}
We can show that the increment term in \eqref{cusum_recur}, i.e.,
\[(u^\intercal x_t)^2 - \sigma^2\left(1+\frac{1}{\rho}\right)\log(1+\rho),
\] has the following property: its expected value is negative under the pre-change and positive under the post-change probability measure. The proof relies on a simple argument based on Jensen's inequality \citep{rudin1987real}. Due to this property, before the change, the CUSUM statistics $S_t$ will oscillate near $0$ while it will exhibit, on average, a positive drift after the occurrence of the change forcing it, eventually, to hit or exceed the threshold.
\end{remark}

\subsubsection{Subspace-CUSUM procedure}\label{sec:S-CUSUM}

Usually the subspace $u$ and SNR $\rho$ are unknown. In this case it is impossible to form the exact CUSUM statistic depicted in \eqref{cusum_recur}. One option is to estimate the unknown parameters and substitute them back into the likelihood function. Here we propose to estimate only $u$ and introduce a new {\it drift} parameter $d$ which plays the same role as $\sigma^2(1+1/\rho)\log(1+\rho)$, this leads to the following Subspace-CUSUM update
\begin{equation}
\mathcal S_t = (\mathcal S_{t-1})^{+} + (\hat{u}_{t+w}^\intercal x_t)^2 - d.
\label{sscusum_update}
\end{equation}
To apply \eqref{sscusum_update}, we need to specify $d$ and of course provide the estimate $\hat u_{t+w}$. Regarding the latter we simply use the \textit{unit-norm} eigenvector corresponding to the largest eigenvalue of $\hat \Sigma_{t+w,w}$ depicted in \eqref{sampleCov}. We denote the estimator of $u$ as $\hat{u}_{t+w}$ because at time $t$ the estimate will rely on the data $x_{t+1},\ldots,x_{t+w}$ that are in the ``future'' of $t$. Practically, this is always possible by properly delaying our data by $w$ samples. Stopping occurs similarly to CUSUM, that is
\begin{equation}\label{sscusum_stop}
T_{\rm SC} = \inf\{t>0: \mathcal S_t \geq b\}.
\end{equation}
Of course, in order to be fair, at the time of stopping we must make the appropriate correction, namely if $\mathcal{S}_t$ exceeds the threshold at $t$ for the first time, then the actual stopping takes place at $t+w$. The reason we use estimates based on ``future'' data is to make $x_t$ and $\hat{u}_{t+w}$ {\it independent} which in turn will help us decide what is the appropriate choice for the drift constant $d$ in Section\,\ref{sec:metric}.

\begin{remark} 
An alternative possibility is to use the generalized likelihood ratio (GLR) statistic, where both $\rho$ and $u$ are estimated for each possible change location $\kappa$. The GLR statistic is 
\[
\max_{\kappa<t} \left\{-\frac{t-\kappa}{2} \log(1+\hat{\rho}_{\kappa,t}) + \frac1{2\sigma^2}\frac{\hat{\rho}_{\kappa,t}}{1+\hat{\rho}_{\kappa,t}}\sum_{i = \kappa+1}^t (\hat{u}_{\kappa,t}^\intercal x_i )^2 \right\},
\] 
where $\hat{\rho}_{\kappa,t}$, $\hat{u}_{\kappa,t}$ are estimated from samples $\{x_i\}_{i = \kappa+1}^t$. However, this computation is more intensive since there is no recursive implementation for the GLR statistic, furthermore it requires growing memory. There are finite memory versions \citep{lai1999efficient}, however are equally complicated in their implementation. Therefore, we do not consider the GLR statistic in this paper.
\end{remark}

\section{THEORETICAL ANALYSIS} \label{sec:theo}
To fairly compare the detection procedures discussed in the previous section, we need to calibrate them properly. The calibration process must be consistent with the performance measure we are interested in.
For a given stopping time $T$ we measure false alarms through the {\it average run length} (ARL) expressed with $\mathbb{E}_\infty[T]$. For the detection capability of $T$ we use the {\it worst-case expected detection delay} (EDD) defined in \cite{lorden1971procedures}
\begin{equation}
\sup\limits_{\tau \geq 0} \supess \mathbb{E}_\tau[ (T - \tau)^{+} | T>\tau, x_1, \ldots, x_\tau],
\label{eq:lorden}
\end{equation}
which considers the worst possible data before the change (expressed through the $\supess$) and the worst possible change-time $\tau$.

In this section, we first discuss the scenarios that lead to the worst-case detection delay for the proposed procedures. Then we characterize the ARL and EDD of the Largest-Eigenvalue Shewhart chart. In doing so, we will also introduce some of the mathematical tools that can be used for the analysis of Subspace-CUSUM. The theoretical characterization of ARL is very important because it can serve as a guideline on how to choose the threshold $b$ used in the detection procedure. Without theoretical analysis, people usually use Monte Carlo simulation to set the threshold, which can be time-consuming when the problem structure is complicated. Therefore a theoretical way to choose the threshold can be beneficial, especially for online change-point detection where computational efficiency is of great importance. 


\subsection{Worst-Case EDD}\label{sec:worstedd}
We now consider scenarios that lead to the worst-case detection delay. For the Largest-Eigenvalue Shewhart chart, assume $1\leq t-w+1\leq\tau< t$. Since for the detection we use $\lambda_{\max} (\hat \Sigma_{t,w})$ and compare it to a threshold, it is clear that the worst-case data before $\tau$ are the ones that will make $\lambda_{\max}$ as small as possible. We observe that
\begin{equation*}
\begin{aligned}
\lambda_{\max} (\hat{\Sigma}_{t,w})&=\lambda_{\max} ( x_{t-w+1} x_{t-w+1}^\intercal + \cdots +
x_{\tau} x_{\tau}^\intercal+\cdots+ x_{t} x_t^\intercal)\\
&\geq\lambda_{\max} ( x_{\tau+1} x_{\tau+1}^\intercal +
\cdots+ x_{t} x_{t}^\intercal) =\lambda_{\max} (\hat{\Sigma}_{t,t-\tau}),
\end{aligned}
\label{eq:worst}
\end{equation*}
which corresponds to the data $x_{t-w+1},\ldots,x_\tau$, before the change, being all equal to zero. In fact, the worst-case scenario at any time instant $\tau$ is equivalent to forgetting all data before and including $\tau$ and restarting the procedure from $\tau+1$ using up to $w$ outer products in the un-normalized sample covariance matrix, exactly as we do when we start at time 0. Due to stationarity, this suggests that we can limit ourselves to the case $\tau=0$ and compute $\mathbb{E}_0[T_{\rm E}]$ and this will constitute the worst-case EDD. Furthermore, the fact that in the beginning we do not normalize with the number of outer products, is beneficial for $T_{\rm E}$ since it improves its ARL. 

We should emphasize that if we do not force the data before the change to become zero and use simulations to evaluate the detector with a change occurring at some time different from 0, then it is possible to arrive at misleading conclusions. Indeed, it is not uncommon for this test to appear outperforming the exact CUSUM test for low ARL values. Of course this is impossible since the exact CUSUM is optimum for \textit{any} ARL in the sense that it minimizes the \textit{worst-case} EDD depicted in \eqref{eq:lorden}.

Let us now consider the worst-case scenario for Subspace-CUSUM. We observe that 
\[
\mathcal S_t = (\mathcal S_{t-1})^{+} + (\hat{u}_{t+w}^\intercal x_t)^2 - d \geq 0 +(\hat{u}_{t+w}^\intercal x_t)^2 - d,
\] 
suggesting that when $\mathcal S_t$ restarts this is the worst it can happen for the detection delay. Therefore, the well-known property of the worst-case scenario in the exact CUSUM carries over to Subspace-CUSUM. Again, because of stationarity, this allows us to fix the change-point time at $\tau=0$. Of course, as mentioned before, because $\hat{u}_{t+w}$ uses data coming from the future of $t$, if our detector stops at some time $t$ (namely when for the first time we experience $\mathcal{S}_t\geq b$) then the \textit{actual} time of stopping must be corrected to $t+w$. A similar correction is not necessary for CUSUM because this test has the exact information for all parameters.

Threshold $b$ is chosen so that the ARL meets a pre-specified value. In practice, $b$ is usually determined by simulation. More specifically, by simulating multiple streams of data from pre-change distribution, we can obtain the average run length for different thresholds. Therefore the threshold can be determined by the simulation results.  

A very convenient tool in accelerating the estimation of ARL (which is usually large) is the usage of the following formula that connects the ARL of CUSUM to the average of the sequential probability ratio test (SPRT) stopping time \citep{siegmund2013sequential}
\begin{equation}
\mathbb{E}_\infty [T_{\rm C}] = \frac{\mathbb{E}_\infty[T_{\rm SPRT}]}{\mathbb{P}_\infty (S_{T_{\rm SPRT}} \ge b)}, 
\label{eq:CUSUM-SPRT}
\end{equation}
where the SPRT stopping time is defined as
\[
T_{\rm SPRT} = \inf\{t>0: S_t \notin [0, b]\}.
\]
The validity of \eqref{eq:CUSUM-SPRT} relies on the CUSUM property that, after each restart, $S_t$ is independent of the data before the time of the restart. Unfortunately, this key characteristic is no longer valid in the proposed Subspace-CUSUM scheme since $\hat{u}_{t+w}$ uses data from the future of $t$. We could, however, argue that this dependence is weak. Indeed, as we have seen in Lemma\,\ref{lem:eigenvector}, each $\hat{u}_t$ is equal to $u$ plus some small random perturbation (estimation error with the power of the order of $1/w$), with these perturbations being practically independent in time. As we observed with numerous simulations, estimating the ARL directly and through \eqref{eq:CUSUM-SPRT} (with $S_t$ replaced by $\mathcal{S}_t$), results in almost indistinguishable values even for moderate window sizes $w$. This suggests that we can use \eqref{eq:CUSUM-SPRT} to estimate the ARL of the Subspace-CUSUM as well. As we mentioned, in the final result, we need to add $w$ to account for the future data used by the estimate $\hat{u}_{t+w}$.

\subsection{Analysis of Largest-Eigenvalue Shewhart Chart}

\subsubsection{Approximate ARL of Largest-Eigenvalue Shewhart chart}

In this section, we first introduce some connection with random matrix theory, which are the building blocks for the theoretical derivation. Then we provide the approximation to ARL as a function of threshold $b$ after taking into account the {\it temporal correlation} between detection statistics. The comparison with simulation results shows the high accuracy of our results. 

The study of ARL requires the understanding of the property of the largest eigenvalue under the null hypothesis, i.e., the samples are i.i.d. Gaussian random vectors with zero-mean and identity covariance matrix. 
To characterize the distribution of the largest eigenvalue, \cite{johnstone2001distribution} uses the  Tracy-Widom law \citep{TracyWidom96}. Define the center constant $\mu_{w,k}$ and scaling constant $\sigma_{w,k}$
\begin{equation}
\begin{aligned}
 \mu_{w,k}&=\big(\sqrt{w-1}+\sqrt{k}\big)^2, \\
 \sigma_{w,k}&=\big(\sqrt{w-1}+\sqrt{k}\big)\left(\frac1{\sqrt{w-1}}+\frac1{\sqrt{k}}\right)^{1/3}. 
 \end{aligned}
 \label{eq:parameter}
 \end{equation} 
 If $k/w \rightarrow \gamma <1$, then the centered and scaled largest eigenvalue converges in distribution to a random variable $W_1$ with the so-called Tracy-Widom law of order one \citep{johnstone2001distribution}:
\begin{equation}
 \frac{\lambda_{\max}(\hat{\Sigma}_{w})-\mu_{w,k}}{\sigma_{w,k}} \rightarrow W_1.
 \label{eq:tw} 
 \end{equation}
The Tracy-Widom law can be described in terms of a partial differential equation and the Airy function, and its tail can be computed numerically (using for example the R-package {\small \textsf{RMTstat}}).

\begin{remark}[Connection with random matrix theory]
There has been an extensive literature on the distribution of the largest eigenvalue of the sample covariance matrix, see, e.g., \cite{johnstone2001distribution, yin1988limit, baik2006eigenvalues, jiang2017rare}. The so-called {\it bulk} \citep{edelman2013random} results are typically used for eigenvalue distributions. It treats a continuum of eigenvalues, and the {\it extremes}, which are the (first few) largest and smallest eigenvalues. Assume there are $w$ samples which are $k$-dimensional Gaussian random vectors with zero-mean and identity covariance matrix. Let $\hat{\Sigma}_w=\sum_{i=1}^w x_i x_i^\intercal$ denote the un-normalized sample covariance matrix. If $k/w \rightarrow \gamma >0$, the largest eigenvalue of the sample covariance matrix converges to $w(1+\sqrt{\gamma})^2$ almost surely \citep{geman1980limit}. Here we use the Tracy-Widom law to characterize its limiting distribution and tail probabilities.
\end{remark}

If we ignore the temporal correlation of the largest eigenvalues produced by the sliding window, we can obtain a simple approximation for the ARL. If we call $p = \mathbb{P}_\infty(\lambda_{\max}(\hat{\Sigma}_{t,w} ) > b)$ for $t \geq w$ then the probability to stop at $t$ is geometric and it is easy to see that the ARL can be expressed as $1/p$. We note that to obtain this result,  we must assume that $\mathbb{P}_\infty(\lambda_{\max}(\hat{\Sigma}_{t,w} ) > b)=p$ for $t<w$ as well, which is clearly not true. Since for $t<w$ the un-normalized sample covariance has less than $w$ terms, the corresponding probability is smaller than $p$. This suggests that $1/p$ is a lower bound to the ARL while $w+1/p$ an upper bound. If $w\ll 1/p$, where $\ll$ means much less than, then approximating the ARL with $1/p$ is quite acceptable. We can use the Tracy-Widom law to obtain an asymptotic expression relating the ARL with the threshold $b$. The desired formula is depicted in the following theorem.

\begin{theorem}[Approximation of ARL by ignoring temporal correlation]
For any $0<p\ll1$ we have $\mathbb{E}_\infty[T_{\rm E}]\approx 1/p$, if we select 
\begin{equation}
b = \sigma_{w,k}b_p + \mu_{w,k}, 
\label{eq:threshold1}
\end{equation}
where $b_p$ denotes the $p$-upper-percentage point of $W_1$ namely $\mathbb{P}(W_1\geq b_p)=p$.
\label{ARL1}
\end{theorem}

Now we aim to capture the temporal correlation between detection statistics due to overlapping time windows. We leverage a proof technique developed in \cite{SiegmundYakirZhang2010}, which can obtain satisfactory approximation for the tail probability of the maximum of a random field. 
\begin{figure}[!b]
\vspace{-0.1in}
\centerline{
\includegraphics[width = 0.6\textwidth]{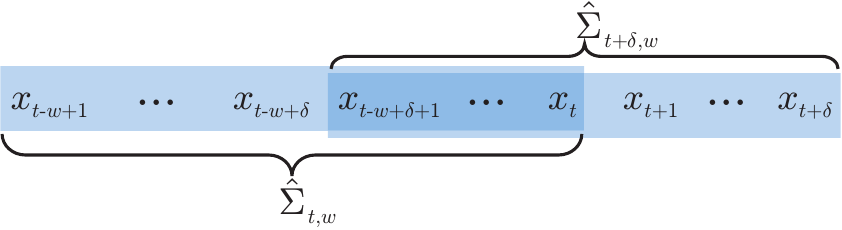} 
}
\vspace{-0.1in}
\caption{Illustration of the temporal correlation between largest eigenvalues, here $\delta \in \mathbb{Z}^+$.}
\label{overlapping}
\end{figure}

Figure\,\ref{overlapping} illustrates the overlap of two sample covariance matrices and provides necessary notation. For each $\hat{\Sigma}_{t,w}$,  define $Z_{t}=\lambda_{\max}(\hat{\Sigma}_{t,w})$. We note that for any given $M > 0$,
\[
\mathbb{P}_\infty(T\leq M) = \mathbb{P}_\infty\left(\max_{1\leq t\leq M} Z_t  \geq b \right),
\]
which is the max over a set of correlated variables $\{ Z_t\}_{t=1}^M$. 
Capturing the temporal dependence of $\{Z_t\}$ is challenging. Below, we assume the dimension $k$ and the window size $w$ are fixed, and consider local covariance structure of the detection statistic when they only non-overlap at a small shift $\delta$ relative to the window size, i.e., $\delta/w$ is small.
By leveraging the properties of the local approximation, we can obtain an asymptotic approximation using the localization theorem \citep{SiegmundYakirZhang2010}.  Define a special function $\nu(\cdot)$ which is closely related to the Laplace transform of the overshoot over the boundary of a random walk \citep{siegmund2007statistics}
\begin{equation}\label{eq:overshot}
\nu(x) \approx \frac{\frac{2}{x}[\Phi\big(\frac{x}{2}\big)-0.5]}{\frac{x}{2}\Phi\big(\frac{x}{2}\big)+\phi\big(\frac{x}{2}\big)}, 
\end{equation}
where $\phi(x)$ and $\Phi(x)$ are the PDF and cumulative distribution function (CDF) of the standard normal distribution $\mathcal{N}(0,1)$. Then we have the following results.
\begin{theorem}[ARL of Largest-Eigenvalue Shewhart chart]
For large values of $b$ we can write
\begin{equation}
\mathbb{E}_\infty[T_{\rm E}] = \left[b'\phi(b')\beta_{k,w}\nu(b'\sqrt{2\beta_{k,w}/w})/w\right]^{-1}\big(1+o(1)\big),
\label{eq:ARL}
\end{equation}
where \[
\beta_{k,w} = 1+\frac{\left(1+c_1 k^{-\frac16}/\sqrt{w}\right)\left(2+c_1k^{-\frac16}/\sqrt{w}\right)}{c_2^2 k^{-\frac13}/w}, \quad b' = \frac{b - (\mu_{w,k} + \sigma_{w,k}c_1)}{\sigma_{w,k}c_2},\]
with $c_1 = \mathbb{E}[W_1]=-1.21$ and $c_2=\sqrt{{\rm Var}(W_1)}= 1.27$.
\label{ARL2}
\end{theorem}
We perform simulations to verify the accuracy of the threshold values obtained without and with considering the temporal correlation (Theorem\,\ref{ARL1} and Theorem\,\ref{ARL2}, respectively). The results are shown in Table\,\ref{eq:threshold}. Compared with the thresholds obtained from Monte Carlo simulation, we find that the threshold, when temporal correlation \eqref{eq:ARL} is taken into account, is more accurate than its counterpart obtained by using the Tracy-Widom law \eqref{eq:threshold1}.

\begin{table}[!b]
\begin{center}
\vspace{-0.2in}
\caption{Comparison of the threshold $b$ obtained from simulations and using the approximations ignoring the correlation in \eqref{eq:threshold1}, and considering the correlation in \eqref{eq:ARL}. Window length $w=200$, dimension $k = 10$. The numbers shown are $b/w$. Approximations that are closer to simulation values are indicated in boldface.}
\vspace{0.2in}
\begin{tabular}{|c|c|c|c|c|c|c|}
\hline
Target ARL & 5k  &  10k  &  20k  & 30k &  40k & 50k\\
\hline
\hline
Simulation   &    1.633  &  1.661 &   1.688   & 1.702   & 1.713  &  1.722   \\
\hline
Approx \eqref{eq:threshold1} & 1.738 & 1.763 & 1.787 & 1.800 &  1.809 & 1.816\\
\hline
Approx \eqref{eq:ARL} & {\bf 1.699}  &  {\bf 1.713}  &  {\bf 1.727}   & {\bf 1.735} &   {\bf 1.740}  &  {\bf 1.744} \\
\hline
\end{tabular}
\label{eq:threshold}
\end{center}
\end{table}

\subsubsection{Lower-bound to EDD of Largest-Eigenvalue Shewhart chart}
We now focus on the detection performance and present a tight lower bound for the EDD of the Largest-Eigenvalue Shewhart chart. The results are based on a known result for CUSUM \citep{siegmund2013sequential} and requires the derivation of the Kullback-Leibler divergence for our problem. 
\begin{theorem}
For large values of $b$ we have
\begin{equation}
\mathbb{E}_0[T_{\rm E}] \geq 2\frac{b^\prime + e^{-b^\prime}-1}{-\log(1+\rho)+\rho}\big(1+ o(1)\big),
\label{eq:EDD}
\end{equation}
where
\[
b' = \frac1{2\sigma^2(1+\rho)}\left[b\rho-(1+\rho)\sigma^2\log(1+\rho)\right].
\]
\label{EDD}
\end{theorem}
\begin{figure}[!b]
\vspace{-0.1in}
\centerline{\includegraphics[width = 0.5\textwidth]{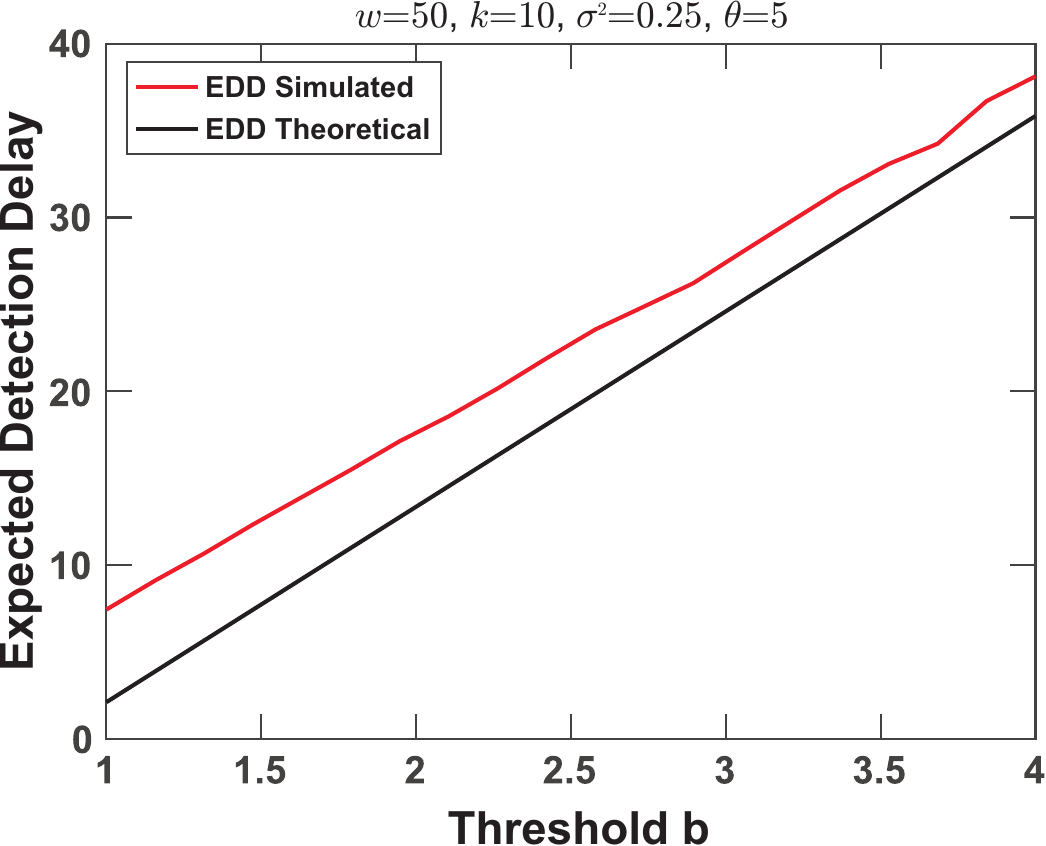}}
\vspace{-0.1in}
\caption{Simulated EDD and lower bound as a function of the threshold $b$.}
\label{fig:EDD}
\end{figure}

Consistent with intuition, in Theorem\,\ref{EDD}, the right-hand-side of \eqref{eq:EDD} is indeed a decreasing function of the SNR $\rho$. Comparing the lower bound in Theorem\,\ref{EDD} with simulated average delay, as shown in Figure\,\ref{fig:EDD}, we can show that in the regime of small detection delay (which is the main regime of interest), the lower bound serves as a reasonably good approximation.

\subsection{Analysis of Subspace CUSUM} \label{sec:metric} 

In this section, we focus on how to set the drift parameter $d$ for Subspace-CUSUM procedure. This is an important parameter for the Subspace-CUSUM to achieve desired properties of change-point detection algorithms. For the drift parameter $d$ we need the following double inequality to be true
\begin{equation}
\mathbb{E}_\infty[(\hat{u}_{t+w}^\intercal x_t)^2] < d < \mathbb{E}_0[(\hat{u}_{t+w}^\intercal x_t)^2].
\label{expected_d}
\end{equation}
With \eqref{expected_d} we can guarantee that $\mathcal{S}_t$ mimics the behavior of the exact CUSUM statistic $S_t$ mentioned in Remark\,\ref{drift_property_1}, namely, it exhibits a negative drift before and a positive after the change. As we mentioned, the main advantage of using $\hat \Sigma_{t+w,w}$ is that it provides estimates $\hat u_{t+w}$ that are \textit{independent} from $x_t$. This independence property allows for the straightforward computation of the two expectations in \eqref{expected_d} and contributes towards the proper selection of $d$. Note that under the pre-change distribution we can write
\begin{equation}
\mathbb{E}_\infty[(\hat{u}_{t+w}^\intercal x_t)^2] =  \mathbb{E}_\infty[ \hat{u}_{t+w}^\intercal\mathbb{E}_\infty[ x_t x_t^\intercal] \hat{u}_{t+w} ]= \sigma^2\mathbb{E}_\infty [\hat{u}_{t+w}^\intercal\hat{u}_{t+w}]= \sigma^2,
\label{drift_infty}
\end{equation}
where the first equation is due to the independence of $x_t$ and $\hat{u}_{t+w}$, the next one due to $x_t$ having covariance $\sigma^2I_k$ and the last equality due to $\hat{u}_{t+w}$ being of unit-norm.

Under the post-change regime, we need to specify the statistical behavior of $\hat u_{t+w}$ for the computation of $\mathbb{E}_0[(\hat{u}_{t+w}^\intercal x_t)^2]$. We will assume that the window size $w$ is sufficiently large so that Central Limit Theorem (CLT) approximations \citep{anderson1963asymptotic, paul2007asymptotics} are possible for $\hat u_{t+w}$. The required result appears in the next lemma.
\begin{lemma}
\label{lem:eigenvector}
Suppose vectors $x_1,\ldots,x_w$ are of dimension $k$ and follow the distribution  $\mathcal{N}(0,  \sigma^2I_k + \theta u u^\intercal)$. Let $\hat{\varphi}_w$ be the eigenvector corresponding to the largest eigenvalue of the sample covariance matrix $(1/w)(x_1x_1^\intercal+\cdots+x_wx_w^\intercal)$, then, as $w \rightarrow \infty$, we have the following CLT version for $\hat{\varphi}_w$
\[ 
\sqrt{w}(\hat{\varphi}_w - u)  \rightarrow \mathcal N\left(0, \frac{ 1+\rho}{\rho^2}  ( I_k - uu^\intercal)   \right).
\]
\end{lemma}

Lemma\,\ref{lem:eigenvector} provides an asymptotic statistical description of the \textit{un-normalized} estimate of $u$. More precisely it characterizes the estimation error $v_w=\hat{\varphi}_w-u$. In our case we estimate the eigenvector from the matrix $\hat \Sigma_{t+w,w}$ but, as mentioned before, we adopt a \textit{normalized} (unit norm) version $\hat{u}_{t}$. Therefore if we fix $w$ at a sufficiently large value and $v_t$ denotes the estimation error of the un-normalized estimate at time $t$ then, from Lemma\,\ref{lem:eigenvector}, we can deduce
\begin{equation*}\label{hat_u_separate}
 \hat{u}_{t+w} = \frac{\hat{\varphi}_{t+w}}{\|\hat{\varphi}_{t+w}\|}=\frac{u + v_{t+w}}{\left\Vert u + v_{t+w} \right\Vert}, \quad v_{t+w} \sim \mathcal{N}\left(0, \frac{1+\rho}{w\rho^2}(I_k-uu^\intercal)\right). 
 \end{equation*}
Consequently
 \begin{equation}\label{lower_bound}
\mathbb{E}_0\left[(\hat{u}_{t+w}^\intercal x_t)^2\right]  =  \sigma^2(1+\rho)\left( 1 - \frac{k-1}{w\rho} +o\left(\frac{1}{w}\right)\right), 
 \end{equation}
with the $o(\cdot)$ term being negligible compared to the other two when $k/w\ll1$, where $a=o(b)$ denotes that $a/b\rightarrow0$.

Consider now the case where $\rho$ is {\it unknown} but exceeds some pre-set minimal SNR $\rho_{\min}$. From the above derivation, given the worst-case SNR and an estimation for the noise variance $\hat{\sigma}^2$, we can give a lower bound for $\mathbb{E}_0[(\hat{u}_{t+w}^\intercal x_t)^2]$. Consequently, the drift $d$ can be anything between 
$\hat{\sigma}^2$
  and 
  $\hat{\sigma}^2(1+\rho_{\min}) (1-  (k-1)/(w\rho_{\min}) )$ where, we observe, that the latter quantity exceeds $\hat{\sigma}^2$ when $w > (k-1)(1+\rho_{\min})/\rho_{\min}^2$. Below, for simplicity, for $d$ we use the average of the two bounds. 
 It is worthwhile mentioning that the lower bound \eqref{drift_infty} and upper bound \eqref{lower_bound} are derived based on the assumption that the window size $w$ is large enough. 

\begin{remark}[Monte Carlo simulation to choose the threshold]\label{rem:monte}
Alternatively, and in particular when $w$ does not satisfy $w\gg k$, we can estimate $\mathbb{E}_0[(\hat{u}_{t+w}^\intercal x_t)^2]$ by Monto Carlo simulation. 
This method requires: (i) estimating the noise level $\hat{\sigma}^2$, which can be obtained from training data without a change-point; (ii) the pre-set worst-case SNR $\rho_{\min}$; (iii) a unit norm vector $u_0$ that is generated randomly. Under the nominal regime we have $\mathbb{E}_\infty[(\hat{u}_{t+w}^\intercal x_t)^2]=\hat{\sigma}^2$. Under the alternative $\mathbb{E}_0[(\hat{u}_{t+w}^\intercal x_t)^2]$ depends only on the SNR $\rho$ as shown in \eqref{lower_bound}. We can therefore simulate the worst-case scenario $\rho_{\min}$ using the randomly generated vector $u_0$ by generating samples from the distribution $\mathcal{N}(0,  \hat{\sigma}^2I_k + \rho_{\min}u_0u_0^\intercal )$.  
\end{remark}

Even though the average of the update in \eqref{sscusum_update} does not depend on true subspace $u$, the computation of the test statistic $\mathcal{S}_t$ \eqref{sscusum_update} requires the estimate $\hat{u}_{t+w}$ of the eigenvector. This can be accomplished by applying the singular value decomposition (SVD) (or the power method \citep{mises1929praktische}) on the un-normalized sample covariance matrix $\hat \Sigma_{t+w,w}$.

\section{SIMULATION STUDY} \label{sec:numerical}
In this section, numerical results are presented to compare the proposed detection procedures. The tests are first
applied to synthetic data, where the performance of the Subspace-CUSUM and Largest-Eigenvalue Shewhart chart are compared against the CUSUM optimum performance. Then the performance of Subspace-CUSUM is optimized by selecting the most appropriate window size.

\subsection{Performance Comparison}

Simulation studies are performed to compare the Largest-Eigenvalue Shewhart chart and the Subspace CUSUM procedure. The exact CUSUM procedure with all parameters known is chosen as the baseline and gives the minimal detection delay to all detection procedures. 

\begin{figure}[!b]
\centering
\begin{tabular}{ccc}
\includegraphics[width = 0.3\textwidth]{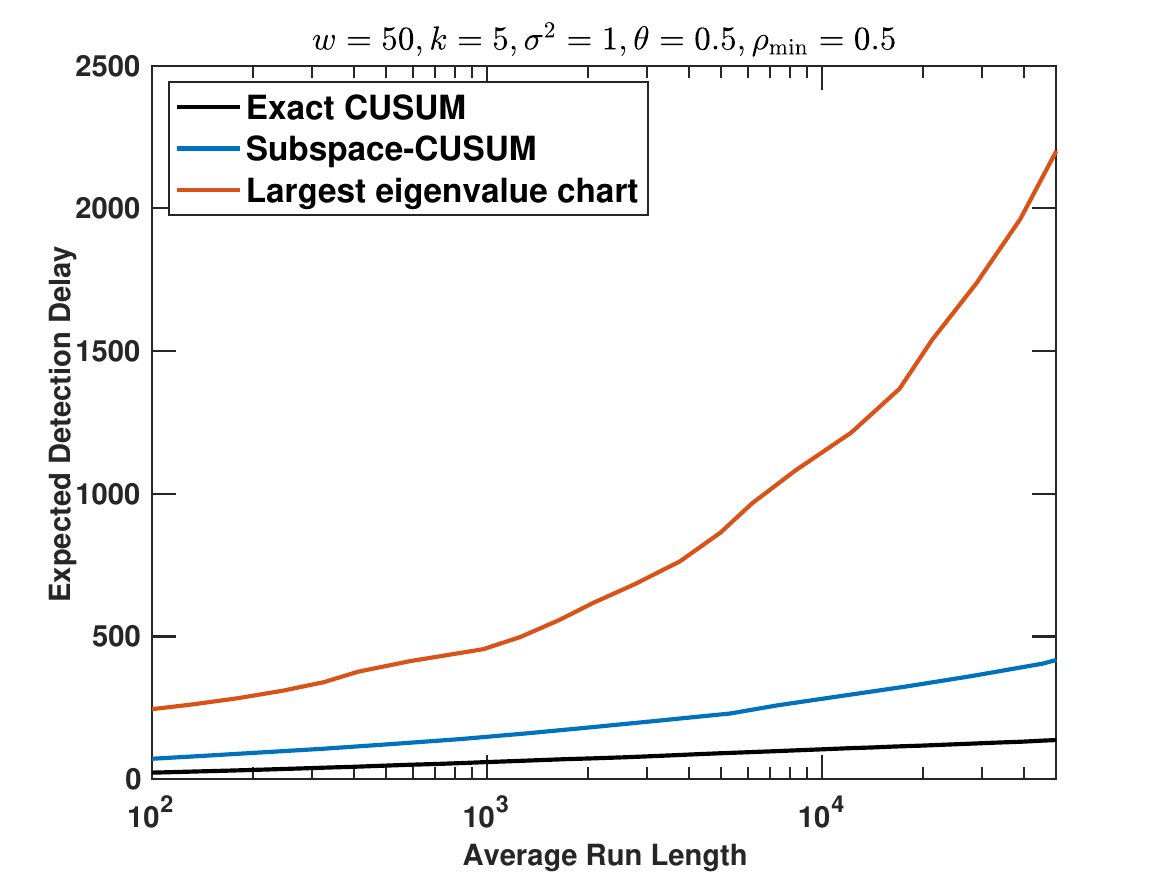} & \includegraphics[width = 0.3\textwidth]{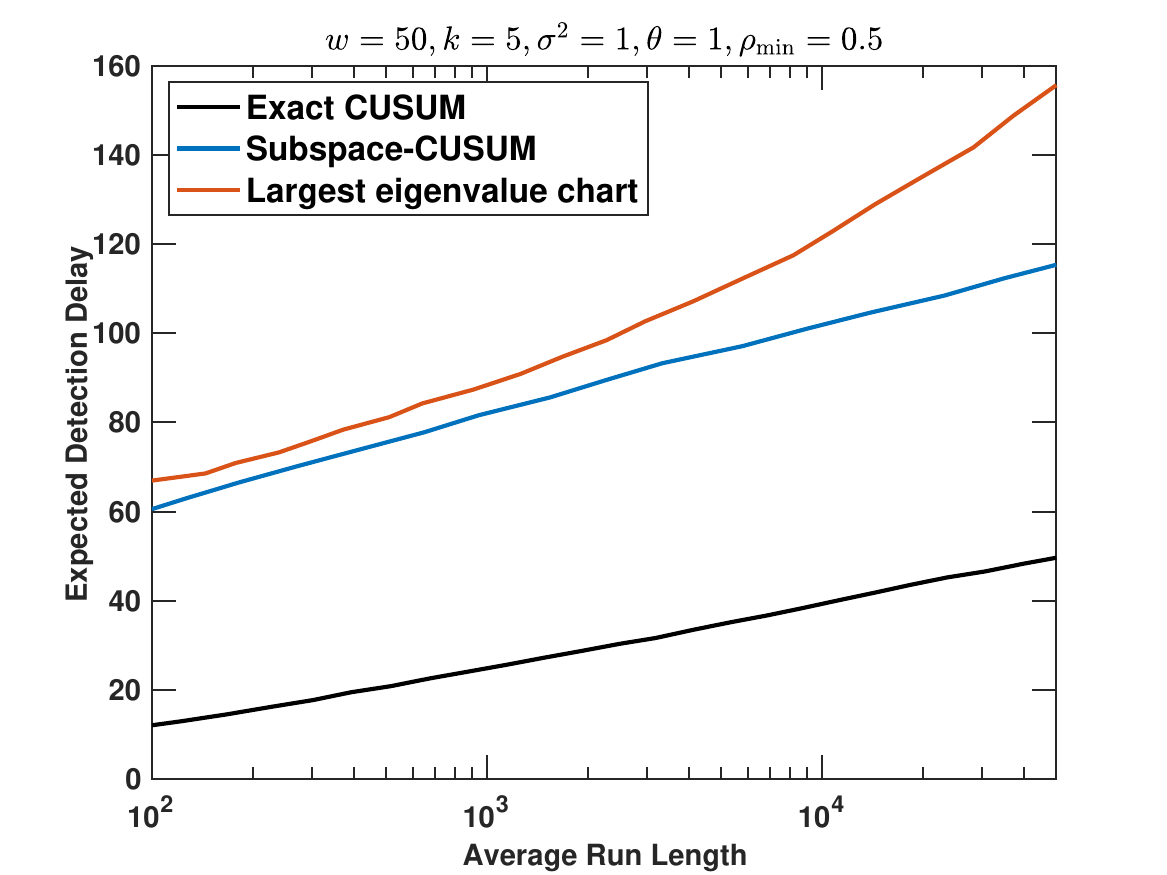}  & 
\includegraphics[width = 0.3\textwidth]{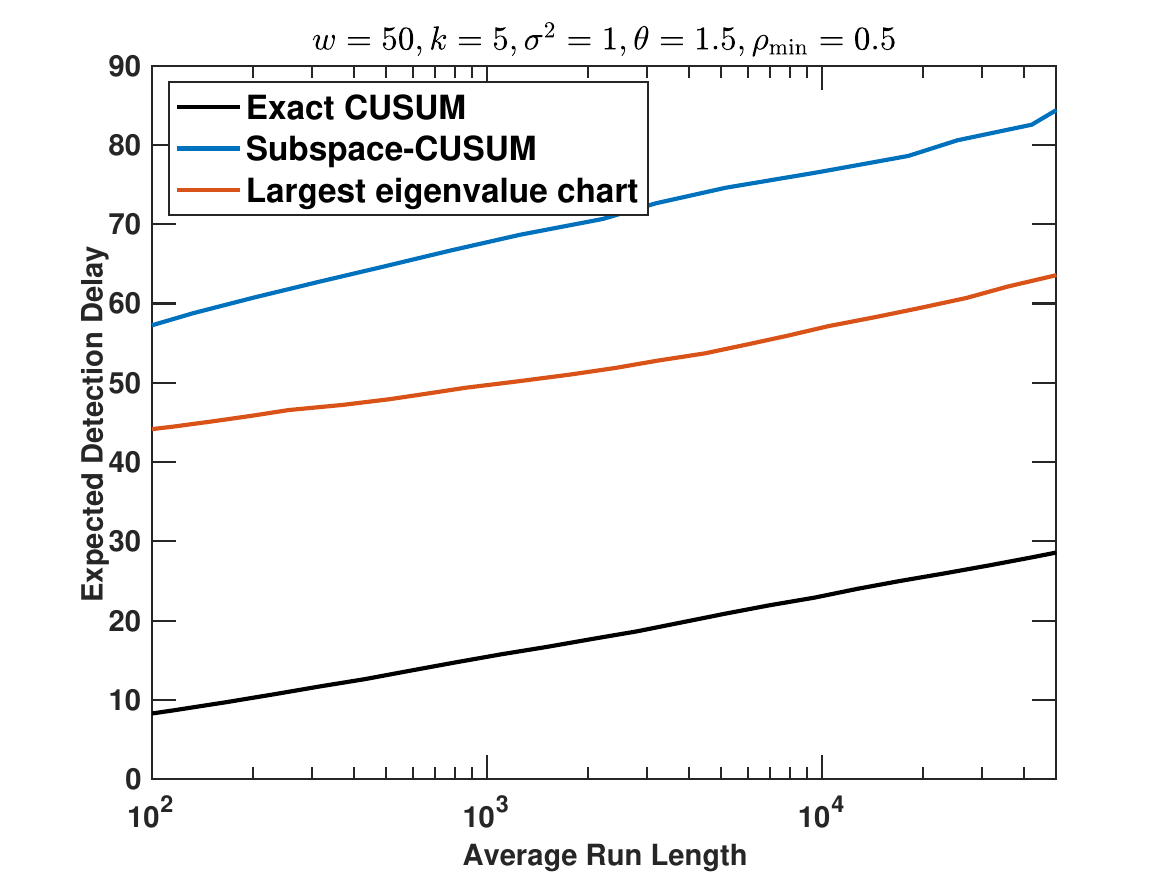}  \\ 
(a) $\theta = 0.5$ & (b) $\theta=1$ & (c) $\theta=1.5 $\\
\end{tabular}
\caption{Comparison of Subspace-CUSUM and the Largest-Eigenvalue Shewhart chart, fixed window size $w=50$. Baseline: Exact CUSUM (optimal).}
\label{fig:eddcompare1}
\end{figure}

Figure\,\ref{fig:eddcompare1} depicts EDD versus log-ARL for parameter values $k=5$, $\sigma^2 = 1$, $w=50$ and three different levels of signal strength (SNR): $\theta=0.5$, $\theta=1$, and $\theta=1.5$. For fair comparison, the SNR lower bound is set to be a constant $\rho_{\min}=0.5$ in all three scenarios. 

The threshold for each procedure is determined using the pre-set lower bound $\rho_{\min}$ as discussed in Remark\,\ref{rem:monte}. In Figure\,\ref{fig:eddcompare1}, the black line corresponds to the exact CUSUM procedure, which is clearly the best and it lies below the other curves. 
Subspace-CUSUM has much smaller EDD than the Largest-Eigenvalue Shewhart chart, and the difference increases with increasing ARL for SNR $\theta= 0.5$ and $\theta= 1$. However, when the signal is stronger ($\theta= 1.5$), the Largest-Eigenvalue Shewhart chart outperforms the Subspace-CUSUM as shown in Figure\,\ref{fig:eddcompare1} (c). 
This is consistent with previous research findings in \cite{neuburger2017comparison} that Shewhart charts are more efficient when detecting strong signals, while the CUSUM-type charts can detect weak signals more quickly due to its cumulative structure.

\subsection{Optimal Window Size}
\begin{figure}[!b]
\begin{center}
\begin{tabular}{cc}
\includegraphics[width = 0.4\textwidth]{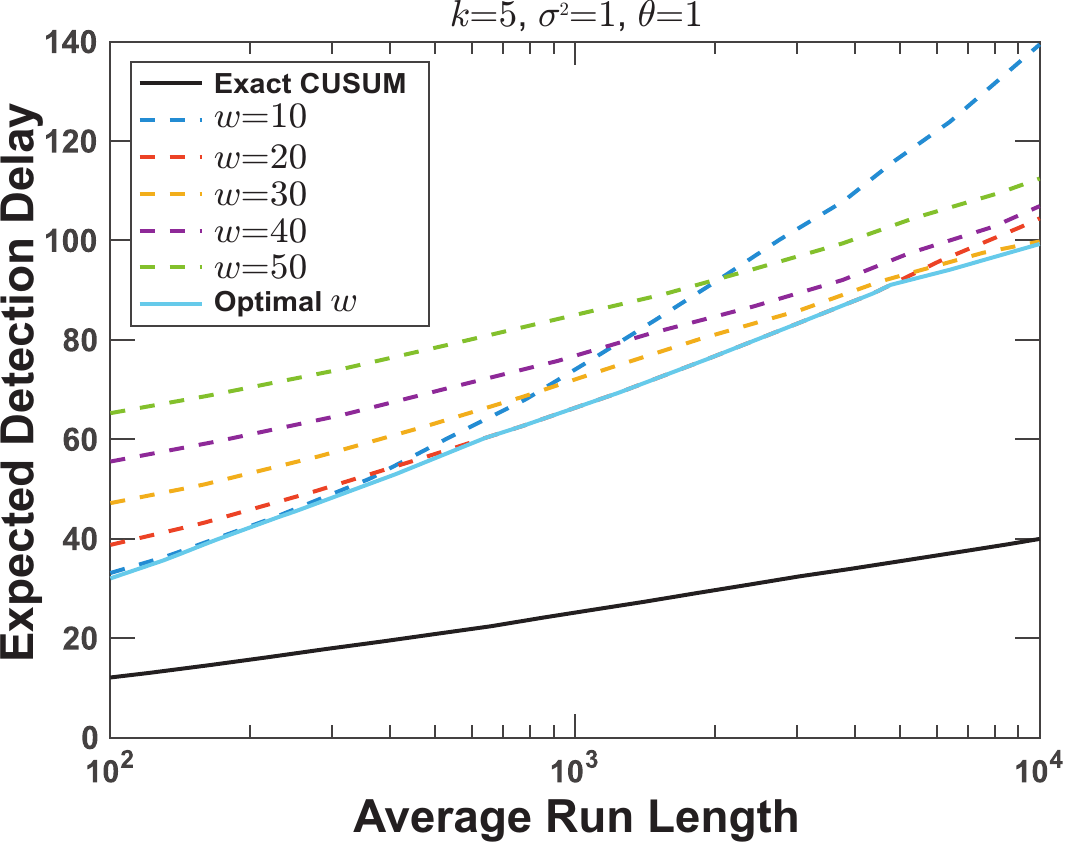} & \includegraphics[width = 0.4\textwidth]{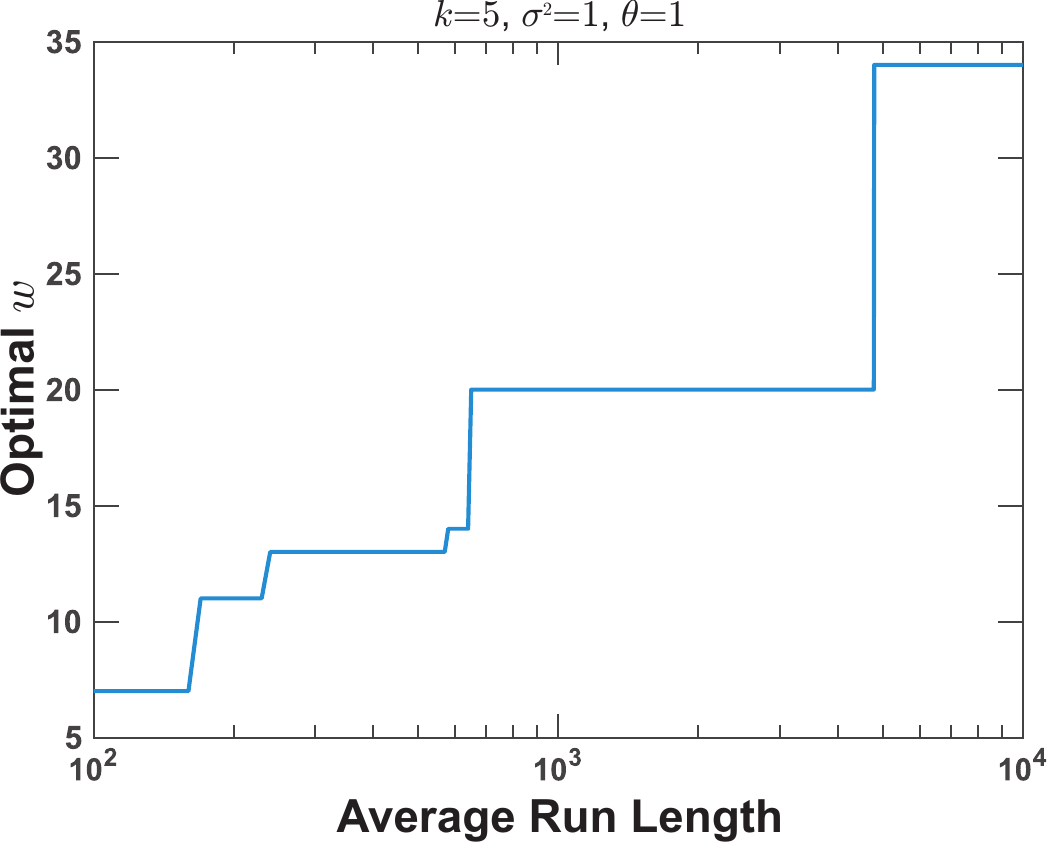}  \\
(a) Minimal EDD & (b): Optimal window size \\
\end{tabular}
\caption{ (a): Minimal EDD vs ARL among window sizes from 1 to 50; (b): Optimal window size resulting in smallest EDD.}
\label{fig:eddw}
\end{center}
\vspace{-0.2in}
\end{figure}
We also consider the EDD-ARL curve, where $w$ is optimized to minimize the detection delay at every ARL. We first compute the EDD for window sizes $w=1,2,\ldots,50,$ given each ARL value. Then we plot in Figure\,\ref{fig:eddw} (a) the lower envelope of EDDs corresponding to the optimal EDD achieved by varying $w$. We also plot the optimal value of $w$ as a function of ARL in Figure\,\ref{fig:eddw} (b). Even though the best EDD of the Subspace-CUSUM is diverging from the performance enjoyed by CUSUM, this divergence we believe is slower than the increase of the optimum CUSUM EDD. One of the goals in the future publication regarding the analysis of Subspace-CUSUM is to show that this is indeed the case, which in turn will demonstrate that this detection structure is first-order asymptotically optimum.

\section{REAL DATA EXAMPLES}\label{sec:data}
In this section, we show how to apply the proposed methods to real data problems and demonstrate the performance using two real datasets; one is the human gesture detection dataset, the other is a seismic dataset. It is worth mentioning that the model formulation is a fundamental problem in high dimensional problems, and the proposed methods are widely applicable to a variety of applications. 

\subsection{Human Gesture Detection}

We apply the proposed method to the sequential posture detection problem using a real dataset: the ``Microsoft Research Cambridge-12 Kinect gesture'' dataset \citep{fothergill2012instructing}. The cross-correlation structure of such multivariate functional data may change over time due to the posture change. \cite{zhang2018dynamic} studies the same dataset from the dynamic subspace learning perspective in the offline setting, our goal is to detect the change-point from sequential observations. This dataset contains 18 sensors. At each time $t$, each sensor records the coordinates in the three-dimensional cartesian coordinate system. Therefore there are 54 attributes in total.

We select a subsequence with a posture change from ``bow'' to ``throw'', and we use the first 250 training samples to estimate the subspace before the change. Figure\,\ref{fig:posture} (a) shows the eigenvalues of the principal component analysis (PCA). We select $r$ leading eigenvectors of the sample covariance matrix as our estimate of the pre-change subspace. For example, when $r=1$, we estimate the pre-change subspace to be a rank one space characterized by the leading eigenvector of the sample covariance matrix of training samples. Then we normalize the observations by multiplying them with a matrix $Q$ that is orthogonal to the pre-change subspace, as discussed in Section\,\ref{sec:formulation}. This enables us to approximate the covariance of pre-change observations by an {\it identity matrix}. Then we apply the proposed Subspace-CUSUM procedure to detect the change. 

The detection statistic is shown in Figure\,\ref{fig:posture} (b,c) for different $r$, the detection statistic indeed increase significantly at the true change-point time (indicated by the red dash line). It also shows that the proposed test performs well not only when $r=1$, but also for $r > 1$ cases, which means that although we focus on the rank one case in the previous discussion, the propose method can be widely used in many problems that involves such low-rank change.
\begin{figure}[!b]
\vspace{-0.1in}
\centering
\begin{tabular}{cc}
\includegraphics[width = 0.46\textwidth]{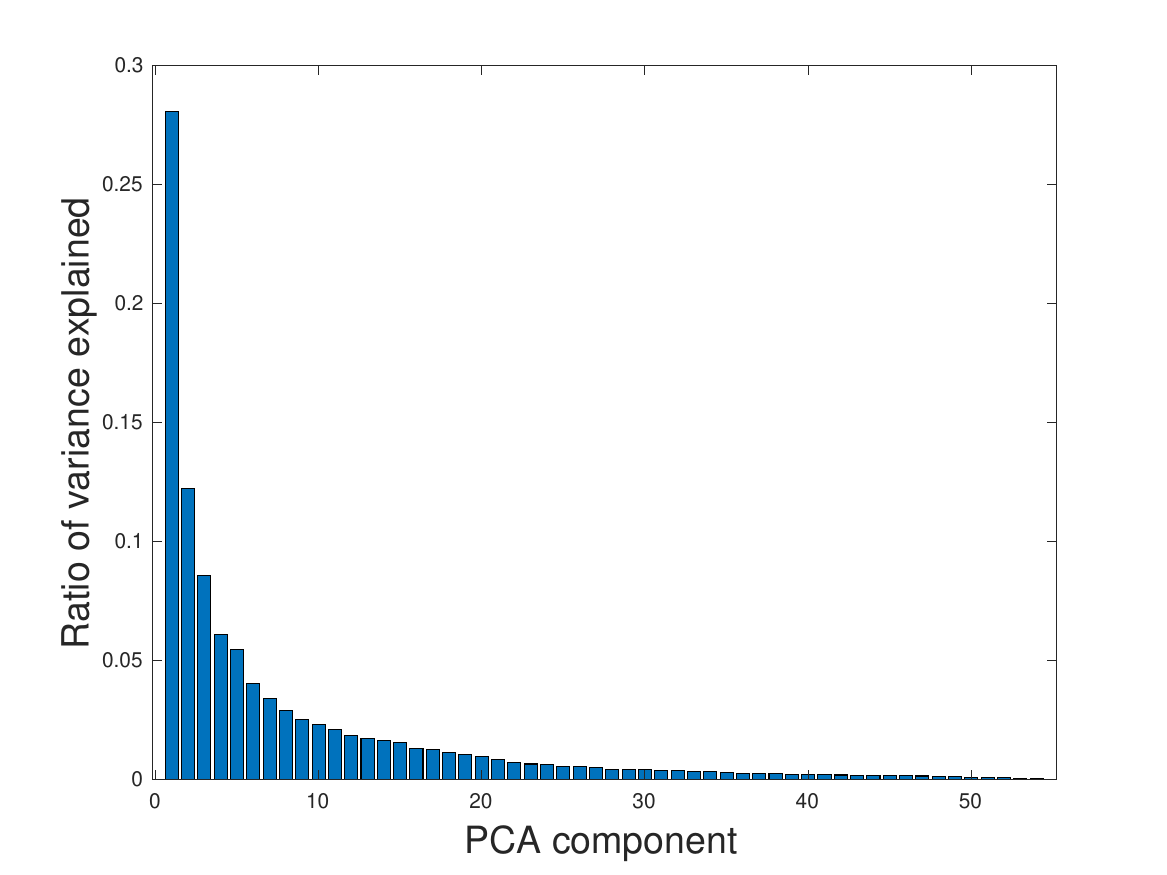} & \includegraphics[width = 0.46\textwidth]{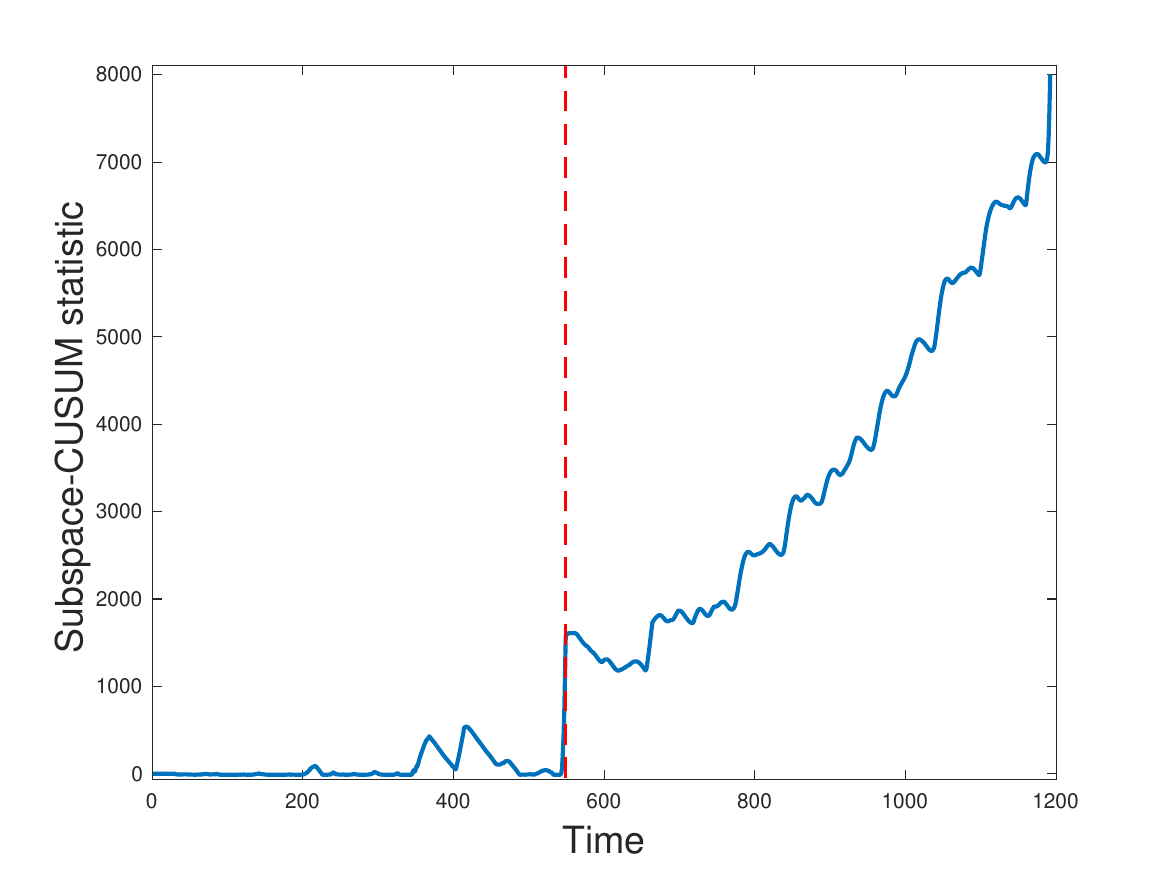} \\
(a) Eigenvalues  &  (b) $r=1$  \\
\includegraphics[width = 0.46\textwidth]{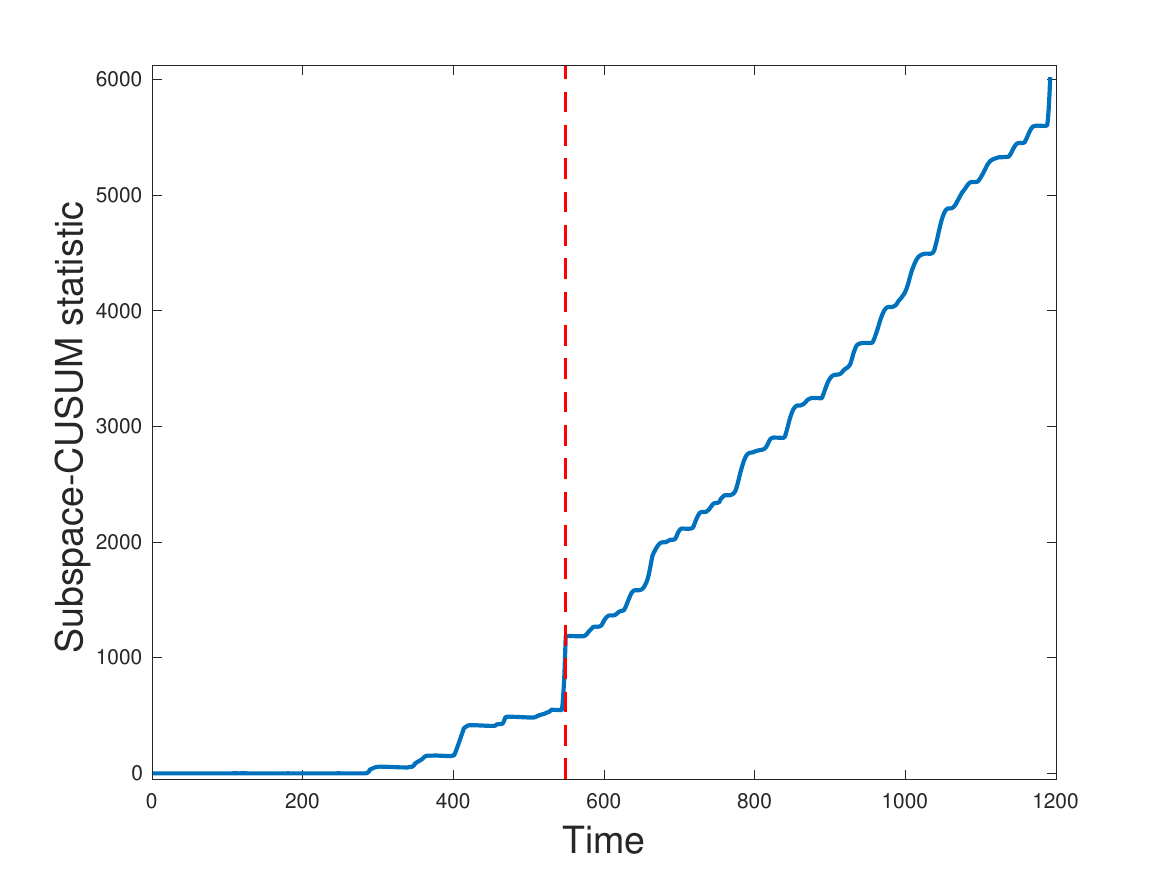} &  \includegraphics[width = 0.46\textwidth]{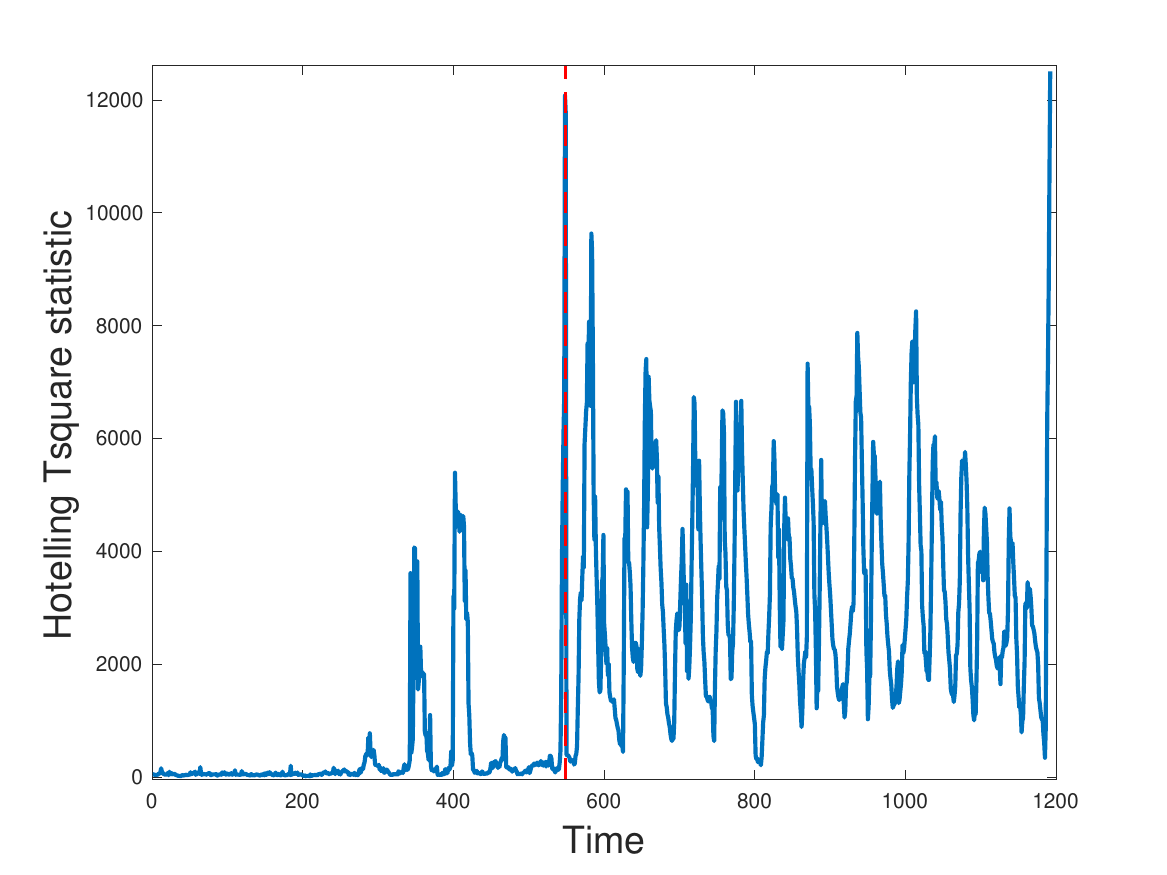}  \\
 (c) $r=10$ & (d) Hotelling $T^2$ \\
\end{tabular}
\vspace{-0.1in}
\caption{(a): PCA Eigenvalues; (b,c): Subspace-CUSUM statistic over time; (d): Hotelling $T^2$ statistic. True change-point indicated by red line.}
\label{fig:posture}
\end{figure}
We also compare the proposed method with Hotelling $T^2$ control chart \citep{hotelling1947multivariate}. We use the same training data to estimate the pre-change mean $\bar\mu$ and covariance matrix $\bar{\Sigma}$, and then construct the Hotelling $T^2$ statistics $(x_t - \bar\mu)^T\bar{\Sigma}^{-1}(x_t - \bar\mu)$. As shown in Figure\,\ref{fig:posture} (d), the detection statistic has a much larger vibration then the Subspace-CUSUM procedure and the performance is sensitive to the estimation of $\bar\mu$ and $\bar\Sigma$.

\subsection{Seismic Event Detection}

In this example, we consider a seismic signal detection problem. The goal is to detect micro-earthquakes and tremor-like signals, which are weak signals caused by minor subsurface changes in the earth. 
The tremor signal may be seen at a subset of sensors, and the affected sensors observe a similar waveform corrupted by noise. The tremor signals are not earthquakes, but they are useful for geophysical study and prediction of potential earthquakes. Usually, the tremor signals are challenging to detect using an individual sensor's data; therefore, network detection methods have been developed, which mainly uses covariance information of the data for detection \citep{li2018high}. We will show that network-based detection can be cast as a subspace detection problem.

Assume that we have $N$ sensors. At an unknown onset, the tremor signal may affect all sensors. Let $s(t)$ be the unknown signal waveform , then the signal observed at sensors can be represented as
\begin{equation}\label{model_seismic}
x_{i}(t) = u_i s(t-\tau) + w_{i}(t), \ i =1,2,\ldots, n,
\end{equation}
where $w_{i}(t) \stackrel{\text{iid}}{\sim} \mathcal{N}(0,\sigma^2)$ denotes the random noise, $u_i > 0$ is the unknown {\it deterministic} magnitude of the signal, and $\tau$ is the unknown change-point or the time when the seismic event happens. Here the waveform function $s(t)$ is assumed to be causal, i.e., $s(t)=0, \forall t<0$. Moreover, we suppose the signal waveform at time time follows a zero-mean normal distribution with time-varying variance (vibration), i.e., $s(t) \sim \mathcal N(0,\sigma_t^2)$.  
 
 \begin{figure}[!b]
\centering
\begin{tabular}{cc}
 \includegraphics[width = 0.46\textwidth]{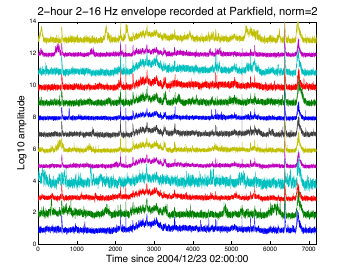}  & \includegraphics[width = 0.46\textwidth]{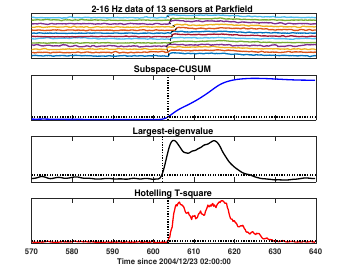} \\   (a) & (b) \\
 \includegraphics[width = 0.46\textwidth]{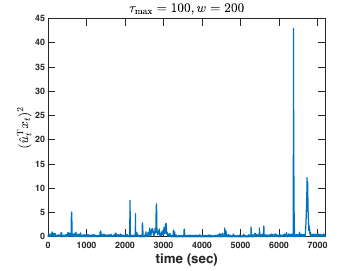}  &   \includegraphics[width = 0.46\textwidth]{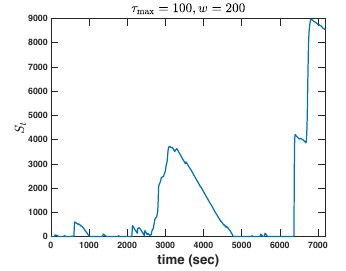} \\
 (c) & (d) \\
 \end{tabular}
\caption{(a) the raw data; (b) comparison of different detection procedures; (c) increment term; (d) Subspace-CUSUM statistic.}
\label{fig:DP1_detection_0_1000s}
\end{figure}

Denote the observation vector $X(t) := [ x_1(t), \ldots, x_n(t)]^\intercal$, and magnitudes $u := [ u_1, \ldots, u_n]^\intercal$. Following \eqref{model_seismic}, we can formulate the problem as follows 
 \begin{equation}
\begin{array}{ll}
X(t) \stackrel{\text{iid}}{\sim} \mathcal{N}(0,\sigma^2I_n), & t = 1,2,\ldots,\tau,\\
X(t) \stackrel{\text{iid}}{\sim} \mathcal{N}(0,  \sigma^2I_n+ \sigma_t^2 uu^\intercal), &t = \tau+1,\tau+2,\ldots.
\end{array} 
\label{eq:hypothesis_seismic}
\end{equation}

We apply the proposed methods to a real seismic dataset recorded at Parkfield, California from 2 am to 4 am on 12/23/2004. The raw data contains records at 13 seismic sensors that simultaneously records a continuous stream of data. The frequency of the raw data is 250Hz. In this example, we set the window size $w = 200$, which corresponds to a $0.8s$ time window. For each procedure, we use the data within the first $600$s to find the threshold by controlling the false alarm rate. 

 We apply the proposed Largest-Eigenvalue Shewhart chart and the Subspace-CUSUM procedure. We further compare them with the classic Hotelling $T^2$ procedure based on the estimated sample mean and sample covariance. The results are shown in Figure\,\ref{fig:DP1_detection_0_1000s}(b). Using the detection statistics in Figure\,\ref{fig:DP1_detection_0_1000s}(c,d), we find three main events at the 615, 2127, and 6371 seconds, as well as some continuous vibration during 2500$\sim$3200 seconds. By comparing the detection results with the true seismic event catalog that can be found online at the Northern California Earthquake Data Center, we found that our findings match the three true events at 594, 2124 and 6369 seconds, along with a tremor catalog around 2500$\sim$3180 seconds. The comparison shows that all detection delays are within 20 seconds. Both the Largest-Eigenvalue Shewhart chart and the Subspace-CUSUM procedure work for this dataset effectively.

\section{CONCLUSION AND DISCUSSIONS}\label{sec:conclusion}
In this paper, we study two online detection procedures for detecting the emergence of a spiked covariance model: the Largest-Eigenvalue Shewhart chart and the Subspace-CUSUM control chart. For Subspace-CUSUM, we perform a simultaneous estimate of the required subspace in parallel with its sequential detection. We avoid estimating all unknown parameters by following a worst-case analysis with respect to the subspace power. We were able to derive theoretical expressions for the ARL and an interesting lower bound for the EDD of the Largest-Eigenvalue Shewhart chart. In particular, we were able to handle the correlations resulted from the usage of a sliding window, which is an issue that is not present in the off-line version of the same procedure. For the comparisons of the two proposed detection procedures, we discuss how it is necessary to calibrate each detector so that comparisons are fair. Comparisons were performed using simulated data and real data about human gesture detection and seismic event detection.

\appendix

\section*{Appendix A: Proof of Lemma\,\ref{lem:eigenvector}}
We have the following asymptotic distribution \citep{anderson1963asymptotic}: 
\[
\frac{1}{\sqrt{w}} (\hat{\varphi}_w- u)  \xrightarrow{d} \mathcal{N}\Big(0, \sum\limits_{j = 2}^k \frac{\lambda_1\lambda_j}{(\lambda_1- \lambda_j)^2}\nu_j\nu_j^\intercal\Big),
\]
where $\lambda_j$ are the $j$th largest eigenvalue of the true covariance matrix and $\nu_j$ are the corresponding eigenvector. In our case the true covariance matrix is $\sigma^2I_k + \theta u u^\intercal$, therefore $\lambda_1 = \sigma^2 + \theta$ and $\lambda_j = \sigma^2$ for $j \geq 2$, and $\{\nu_j, j\geq 2 \}$ is a basis of the orthogonal space of $u$. 
Thus we have
\[
\begin{aligned}
&\sum\limits_{j = 2}^k \frac{\lambda_1\lambda_j}{(\lambda_1- \lambda_j)^2}\nu_j\nu_j^\intercal = \frac{\sigma^2(\sigma^2 + \theta)}{\theta^2}\sum\limits_{j = 2}^k \nu_j\nu_j^\intercal = \frac{\sigma^2(\sigma^2 + \theta)}{\theta^2}(  I_k  - u u^\intercal) = \frac{( 1+\rho)}  {\rho^2} ( I_k - uu^\intercal) .
 \end{aligned}
\]
This completes the proof.

\section*{Appendix B: Proof of Theorem\,\ref{ARL2}}
In order to prove Theorem\,\ref{ARL2}, we need the following lemma to characterize the local correlation between largest eigenvalue statistics. 
\begin{customlem}{B.1}[Approximation of local correlation]\label{CovZ}
Let $c_1 = \mathbb{E}[W_1]=-1.21$, $c_2=\sqrt{{\rm Var}(W_1)}= 1.27$, and 
\begin{equation*}
\beta_{k,w} = 1+\frac{\left(1+c_1 k^{-\frac16}/\sqrt{w}\right)\left(2+c_1k^{-\frac16}/\sqrt{w}\right)}{c_2^2 k^{-\frac13}/w}.
\label{corr_approx}
\end{equation*}
Then we have,
\begin{equation*}
{\rm corr}(Z_t,Z_{t+\delta}) \leq 1 - \beta_{k,w}\vartheta+ o(\vartheta),
\label{corr}
\end{equation*}
where $\vartheta = \delta/w$ and ${\rm corr}(X, Y)$ stands for the Pearson's correlation 
\[
{\rm corr}(X, Y) = \frac{\mathbb{E}[X Y] - \mathbb E[X] \mathbb E[Y]}{\sqrt{{\rm Var}(X)}\sqrt{{\rm Var}(Y)}}.
\]
\end{customlem}
\begin{proof}
Under the pre-change measure, $x_t \stackrel{\rm iid}{\sim} \mathcal{N}(0, \sigma^2I_k)$. For $\delta \in \mathbb{Z}^+$, let 
\[
P=\sum\limits_{i=t-w + 1}^{t-w+\delta}x_ix_i^\intercal,  \quad 
Q=\sum\limits_{i=t-w+\delta + 1}^{t} x_ix_i^\intercal,  \quad
R=\sum\limits_{i=t + 1}^{t+\delta} x_ix_i^\intercal.
\]
Then $P$, $Q$ and $R$ are independent random matrices. Now we also want to give a general upper bound for the covariance between $Z_{t}$ and $Z_{t+\delta}$. Then we have
\begin{equation*}\label{eq:covariance_upper_bound}
\begin{aligned}
\mathbb{E}[Z_{t}Z_{t+\delta}]& =\mathbb{E}[\lambda_{\max}(\hat{\Sigma}_{t,w})\lambda_{\max}(\hat{\Sigma}_{t+\delta,w})] = \mathbb{E}[\lambda_{\max}(P+Q)\lambda_{\max}(Q+R)]\\
&\leq\mathbb{E}\left\{ \left[ \lambda_{\max}(P)+\lambda_{\max}(Q)\right] \left[ \lambda_{\max}(Q)+\lambda_{\max}(R)\right] \right\} \\
&=\mathbb{E}[\lambda_{\max}^2(Q)]+\mathbb{E}[\lambda_{\max}(Q)]\mathbb{E}[\lambda_{\max}(R)]\\
&\quad+\mathbb{E}[\lambda_{\max}(P)]\{\mathbb{E}[\lambda_{\max}(Q)]+\mathbb{E}[\lambda_{\max}(R)]\},
\end{aligned}
\end{equation*}
where the inequality is due to the fact that the largest eigenvalue of the sum of two nonnegative definite matrices is upper bounded by the sum of the corresponding largest eigenvalues of the two matrices.
The mean and second-order moments can be computed using the Tracy-Widom law depicted in \eqref{eq:tw}. 

Since $k$ is a fixed constant, we just write $\mu_{n}$ and $\sigma_n$ instead of $\mu_{n,k}$ and $\sigma_{n,k}$ to simplify our notation. We first consider the covariance term $\mathbb{E}[Z_{t}Z_{t+\delta}]$ and decompose it into four parts as following:
$$
\frac1{w^2}\mathbb{E}[Z_{t}Z_{t+\delta}]\leq A+B+C+D,
$$
where
\begin{align*}
A&=\left(\frac{\mu_{w(1-\vartheta)}+c_1\sigma_{w(1-\vartheta)}}{w}\right)^2,\\
B&=\left(\frac{c_2\sigma_{w(1-\vartheta)}}{w}\right)^2,\\
C&=2\left[\frac{\mu_{w(1-\vartheta)}+c_1\sigma_{w(1-\vartheta)}}{w}\right]\left(\frac{\mu_{w\vartheta}+c_1\sigma_{w\vartheta}}{w}\right),\\
D&=\left(\frac{\mu_{w\vartheta}+c_1\sigma_{w\vartheta}}{w}\right)^2.
\end{align*}

First, the common terms $\mu_{w(1-\vartheta)}/w$ and $\sigma_{w(1-\vartheta)}/w$ can be written as
\begin{align*}
\frac{\mu_{w(1-\vartheta)}}{w}&= \frac 1w \left[ \sqrt{w(1-\vartheta)-1}+\sqrt{k}\right]^2 = \frac{w(1-\vartheta)-1}{w} \left[1+\sqrt{\frac{k}{w(1-\vartheta)-1}} \right]^2\\
&~\dot= \frac{w(1-\vartheta)-1}{w} \dot= 1-\vartheta.
\end{align*}
where the second term in the bracket was ignored because $k/w = o(1)$. Moreover, we have
\begin{equation*}
\frac{\sigma_{w(1-\vartheta)}}{w} = \frac 1w \left(\sqrt{w(1-\vartheta)-1}+\sqrt{k}\right) \\
\cdot \left(\frac1{\sqrt{w(1-\vartheta)-1}}+\frac1{\sqrt{k}}\right)^{1/3},
\end{equation*}
after extracting the term $\sqrt{w(1-\vartheta)-1}$ from the first bracket and $1/\sqrt{k}$ from the second bracket, we obtain
\begin{align*}
\frac{\sigma_{w(1-\vartheta)}}{w}  &= \frac{k^{-\frac16}}{w}\sqrt{w(1-\vartheta)-1}\left(1+\sqrt{\frac{k}{w(1-\vartheta)-1}}\right)^{\frac43} \dot= \sqrt{\frac{1-\vartheta}{w}}k^{-\frac16},
\end{align*}
where the second term in the bracket was ignored because $k/w = o(1)$. 
Plug these two terms into the first part  we have:
\begin{align*}
A &= \left(\frac{\mu_{w(1-\vartheta)}+c_1\sigma_{w(1-\vartheta)}}{w}\right)^2 \dot=  \left(1-\vartheta + c_1\sqrt{\frac{1-\vartheta}{w}}k^{-\frac16}\right)^2 \\
&= (1-\vartheta)\left(1-\vartheta+2c_1\frac{k^{-\frac16}}{\sqrt{w}}\sqrt{1-\vartheta}+c_1^2\frac{k^{-\frac13}}{w}\right).
\end{align*}
Since $\vartheta$ is relatively small, $\sqrt{1-\vartheta} = 1 - \frac12\vartheta + o(\vartheta)$ by Taylor expansion. Then the term is approximately
\begin{align*}
A&\dot=  (1-\vartheta)\left(1-\vartheta+2c_1\frac{k^{-\frac16}}{\sqrt{w}}(1-\frac12\vartheta+o(\vartheta))+c_1^2\frac{k^{-\frac13}}{w}\right)\\
& = \left(1+c_1\frac{k^{-\frac16}}{\sqrt{w}}\right)^2 \!\!\!-\!\! \left(1+c_1\frac{k^{-\frac16}}{\sqrt{w}}\right)\!\!\left(2+c_1\frac{k^{-\frac16}}{\sqrt{w}}\right)\vartheta \!+\! o(\vartheta),
\end{align*}
where the higher order terms of $\vartheta$ are included in $o(\vartheta)$. 
Parts $C$ and $D$ can be considered negligible, since $C$ is order $\mathcal O(\vartheta)$ and $D$ is order $o(\vartheta)$. In summary, we have
\begin{align*}
{\rm corr}(Z_t,Z_{t+\delta}) & =\frac{\mathbb{E}[Z_t Z_{t+\delta}] - \mathbb E[Z_t] \mathbb E[Z_{t+\delta}]}{\sqrt{{\rm Var}(Z_t)}\sqrt{{\rm Var}(Z_{t+\delta})}}\\
&\lesssim  \frac1{\left(\frac{c_2k^{-\frac16}}{\sqrt{w}}\right)^2}\Bigg\{\left(1+c_1\frac{k^{-\frac16}}{\sqrt{w}}\right)^2  + c_2^2 \frac{1-\vartheta}{w}k^{-\frac13} -  \left(1+c_1\frac{k^{-\frac16}}{\sqrt{w}}\right)\left(2+c_1\frac{k^{-\frac16}}{\sqrt{w}}\right)\vartheta\\
& \qquad - \left(1+ c_1\frac{k^{-\frac16}}{\sqrt{w}}\right)^2 + o(\vartheta)\Bigg\} \\
& = 1-  \beta_{k,w}\vartheta + o(\vartheta)
\end{align*}
This completes the proof.
\end{proof}

The key of proving Theorem\,\ref{ARL2} is to quantify the tail probability of the detection statistic. However, this probability is very small when the threshold is large \citep{li2015m}. Therefore we use the change-of-measure technique in \cite{SiegmundYakirZhang2010} to recenter the process mean to the threshold, so that the tail probability becomes much higher. First, the detection statistic is standardized by: 
\[ Z'_t = \frac{Z_t - \mathbb{E}_\infty[Z_t]}{\text{Var}_\infty(Z_t)},\]
here $ \mathbb{E}_\infty[Z_t]$ and $\text{Var}_\infty(Z_t)$ depends only on $k$ and $w$, but does not depend on $t$. Then $Z'_t$ has zero mean and unit variance under the $\mathbb{P}_\infty$ measure. We are interested in finding the probability $\mathbb{P}_\infty(T_{\rm E}\leq M) = \mathbb{P}_\infty(\max_{1\le t\le M} Z'_t > b)$. We now prove our proposition in four steps. 

\noindent {\it Step 1. Exponential tilting. } 
Denote the cumulant generating function of $Z'_t$ by
\begin{equation*}
\psi(a) = \log\mathbb{E}_\infty[e^{a Z'_t}].
\end{equation*}
Define a family of new measures
\begin{equation*}
\frac{d\mathbb{P}_t}{d\mathbb{P}_\infty} = \exp\{a Z'_t - \psi(a)\},
\end{equation*}
where $\mathbb{P}_t$ denotes the new measure after the transformation. The new measure takes the form of the exponential family, and $a$ can be viewed as the natural parameter. It can be verified that $\mathbb{P}_t$ is indeed a probability measure 
since 
\begin{equation*}
\int d\mathbb{P}_t = \int \exp\{a Z'_t-\psi(a)\}d\mathbb{P} = 1.
\end{equation*}
It can also be shown that $\dot{\psi}(a)$ is the expected value of $Z'_t$ under $\mathbb{P}_t$, since
$$
\dot{\psi}(a)=\frac{\mathbb{E}_\infty[Z'_te^{a Z'_t}]}{\mathbb{E}_\infty[e^{a Z'_t}]} = \mathbb{E}_\infty[Z'_te^{a Z'_t-\psi(a)}] = \mathbb{E}_t[Z'_t],
$$
and similarly $\ddot{\psi}(a)$ is the variance under the tilted measure. We use the Gaussian approximation for $Z_t'$, then its log moment generating function is $\psi(a) = a^2/2$. We set $a=b$ such that $\dot{\psi}(a) = \mathbb{E}_t[Z'_t] = b$, therefore the tail probability after measure transformation will become much larger. Given this choice, the transformed measure is given by $d\mathbb{P}_t = \exp(bZ'_t - b^2/2)d\mathbb{P}_\infty$. We also define, for each $t$, the log-likelihood ratio $\log(d\mathbb{P}_t/d\mathbb{P}_\infty)$ of the form
\begin{equation*}
\ell_t = b Z'_t - \frac{1}{2}b^2.
\end{equation*}

\noindent {\it Step 2. Change-of-measure by the likelihood ratio identity.} 
Now we convert the original problem of finding the small probability that the maximum of a random field exceeds a large threshold, to another problem: finding an alternative measure under which the event happens with a much higher probability. By likelihood ratio identity, we have:
\begin{equation}\tag{B.1}\label{eq:tail_prob}
\begin{aligned}
 \mathbb{P}_\infty(\max\limits_{1\le m \le M} Z'_m \geq b) 
& = \mathbb{E}_\infty [\mathbbm{1}_{\{ \max\limits_{1\le m \le M} Z'_m \geq b\}}]\nonumber = \mathbb{E}_\infty \left[\frac{\sum_{t=1}^{M}e^{\ell_t}}{\sum_{n=1}^{M}e^{\ell_n}} \cdot \mathbbm{1}_{\{\max\limits_{1\le m \le M} Z'_m \geq b\}}\right]\\
&=  \sum_{t=1}^M \mathbb{E}_\infty \left[\frac{e^{\ell_t}}{\sum_n e^{\ell_n}} \cdot \mathbbm{1}_{\{ \max\limits_{1\le m \le M} Z'_m \geq b\}} \right]\\
&=   \sum_{t=1}^M \mathbb{E}_t \left[\frac{1}{\sum_n e^{\ell_n}} \cdot \mathbbm{1}_{\{ \max\limits_{1\le m \le M} Z'_m \geq b\} }\right]\\
& =  e^{-b^2/2} \sum_{t=1}^M \mathbb{E}_t \left[\frac{M_t}{S_t}e^{-(\tilde{\ell}_t +\log M_t)} \cdot \mathbbm{1}_{\{ \tilde{\ell}_t +\log M_t \geq 0 \} }\right],
\end{aligned}
\end{equation}
where $M_t$ and $S_t$ in the last step is defined as the maximum and sum of likelihood ratio differences as:
\begin{equation*}
M_t = \max_{m \in \{1, \ldots, M\}} e^{\ell_m -\ell_t}, \quad S_t = \sum_{m \in \{1, \ldots, M\}} e^{\ell_m -\ell_t}.
\end{equation*}
And $\tilde{\ell}_t $ is defined as the re-centered likelihood ratio, or the so-called global term:
\begin{equation*}
\tilde{\ell}_t = b(Z'_t -b).
\end{equation*}
The last equation in \eqref{eq:tail_prob} converts the tail probability to a product of two terms: a deterministic term $e^{-b^2/2}$ associated with the large deviation rate, and a sum of expectations under the transformed measures. The expectation involves a product of the ratio $M_t/S_t$ and an exponential function that depends on $\tilde{\ell}_t$, which plays the role of a weight. Under the new measure $\mathbb{P}_t$, $\tilde{\ell}_t$ has zero mean and variance equal to $b^2$ and it dominates the other term $\log M_t$, hence, the probability of exceeding zero is much higher. Next, we characterize the limiting ratio and the other factors precisely, by the localization theorem.

\noindent {\it Step 3. Establish Properties of Local and Global Terms.} 
In \eqref{eq:tail_prob}, our target probability has been decomposed into terms that only depend on (i) the local field $\{\ell_m - \ell_t\}, 1 \le m \le M$, which are the differences between the log-likelihood ratios with parameter $t$ and $m$, and (ii) the global term $\tilde{\ell}_t$, which is the centered and scaled likelihood ratio with parameter $t$. We need to first establish some useful properties of the local field and global term before applying the localization theorem. We will eventually show that the local field and the global term are asymptotically independent.

For the local field $\{\ell_m - \ell_t\}$, let $r_{m,t}$ denote the correlation between $Z_m^\prime$ and $Z_t^\prime$,  then we have
\[
\begin{aligned}
\mathbb{E}_t(\ell_m - \ell_t) &= -b^2(1-r_{m,t}),\\
\mbox{Var}_t(\ell_m - \ell_t) &= 2b^2(1-r_{m,t}),\\
\mbox{Cov}_t(\ell_{m_1} - \ell_t,\ell_{m_2} - \ell_t) &= b^2(1+r_{m_1,m_2} - r_{m_1,t} - r_{m_2,t}).
\end{aligned}
\]
We have Lemma\,\ref{CovZ} to characterize the local correlation, which offers reasonably good approximation for $\mathbb{E}[Z_{t}Z_{t+\delta}]$ and leads to $r_{m,t} \approx 1 - \left|m-t\right|\beta_{k,w}/w$.

Since we assume $Z'_t$ is approximately Gaussian, the local field $\ell_m - \ell_t$ and the global term $\tilde{\ell}_t$ are also approximately Gaussian.
Therefore, when $|\delta|$ is small (i.e., in the neighborhood of zero), we can approximate the local field using a two-sided Gaussian random walk with drift $b^2\beta_{k,w}/w$ and variance of the increment equal to $2b^2\beta_{k,w}/w$:
\begin{equation*}
\ell_{t+\delta} - \ell_t \stackrel{\Delta}{=} b\sqrt{\frac{2\beta_{k,w}}{w}} \sum\limits_{i=1}^{|\delta|} \xi_i - b^2\frac{\beta_{k,w}}{w}\delta, \delta = \pm1,\pm2,\ldots,
\label{eq:local_decomp}
\end{equation*}
where $\xi_i$ are i.i.d. standard normal random variables.

\noindent {\it Step 4. Approximation using localization theorem.} 
From the argument in \cite{siegmund2000tail}, let $\hat M_t$ and $\hat S_t$ denote the maximization and summation restricted to the small neighborhood of $t$. Then they are asymptotically independent of the global term $\tilde{\ell}_t$. Moreover, under the tilted measure,
\begin{equation*}
\mathbb{E}_t[\tilde{\ell}_t] = 0, \quad \text{Var}_t[\tilde{\ell}_t] = b^2.
\end{equation*}
Therefore the density $\mathbb{P}_t(\tilde{\ell}_t)$ can be approximated by $1/\sqrt{2\pi b^2}$ in a neighborhood of radius $o(1/b)$ of zeros. The inner expectation in \eqref{eq:tail_prob} can be approximated as
\[
\mathbb{E}_t \left[\frac{M_t}{S_t}e^{-(\tilde{\ell}_t +\log M_t)} \cdot \mathbbm{1}_{\{ \tilde{\ell}_t +\log M_t \geq 0 \} }\right] \dot= \frac{ \mathbb{E}_t (\hat M_t/\hat S_t )} {b\sqrt{2\pi b^2}}.
\]
By \cite{siegmund2000tail}, the expectation $  \mathbb{E}_t (\hat M_t/\hat S_t ) $ does not depend on $t$ and equals to $b^2\beta_{k,w}\nu(b\sqrt{2\beta_{k,w}/w})/w$ in the asymptotic regime. 
Substitute into \eqref{eq:tail_prob}, we have
\begin{align*}
\mathbb{P}_\infty(T\le M) & = \mathbb{P}_\infty\left( \max_{1\le t\le M} Z'_t > b \right) \\
&= e^{-b^2/2}\sum\limits_{t=1}^M\mathbb{E}_t \left[ \frac {M_t}{S_t}e^{-[\tilde{\ell}_t+\log M_t]}\cdot \mathbbm{1}_{\{ \tilde{\ell}_t+\log M_t\ge0\}} \right]\\
&\dot= Mb\phi(b)\beta_{k,w}\nu(b\sqrt{2\beta_{k,w}/w})/w,
\end{align*}
where $\nu(\cdot)$ is the function defined in \eqref{eq:overshot}. From the above cumulative distribution function, we can approximate $T$ as exponential distribution, yielding the mean value $1/[b\phi(b)\beta_{k,w}\nu(b\sqrt{2\beta_{k,w}/w})/w]$.

Since $Z'_t$ is standardized, here the threshold $b$ need to be converted to the original threshold using a simple formula \[b' = \left[ b - ( \mu_{w,k} + c_1 \sigma_{w,k}) \right]/ (c_2 \sigma_{w,k}).\]
This completes the proof.

\section*{Appendix C: Proof of Theorem\,\ref{EDD}}

We first relate the largest eigenvalue procedure to a CUSUM procedure, note that 
\begin{equation}\tag{C.1}
\lambda_{\max}(\hat{\Sigma}_{t,w}) = \max_{\|q\|=1}
q^\intercal \hat{\Sigma}_{t,w} q.
\label{eq:q}
\end{equation}
For each $q$, we have
\[
q^\intercal \hat{\Sigma}_{t,w} q =  \sum_{i=t-w+1}^t (q^\intercal x_i)^2.
\]
According to the Grothendieck's Inequality \citep{guedon2016community}, the $q$ that attains the maximum in equation \eqref{eq:q} is very close to $u$ under the alternative. Therefore, assuming the optimal $q$ always equals to $u$ will only cause a small error but will bring great convenience to our analysis. 

Now we have under $\mathbb{P}_\infty$, $q^\intercal x_i \sim \mathcal{N}(0, \sigma^2)$ and under $\mathbb{P}_0$, $q^\intercal x_i \sim \mathcal{N}(0, \sigma^2+\theta)$. Let $f_\infty$ denote the pdf of $\mathcal{N}(0, \sigma^2)$ and $f_0$ the pdf of $\mathcal{N}(0, \sigma^2+\theta)$. For each observation $y$, we can derive the one-sample log-likelihood ratio:
\[
\log\frac{f_0(y)}{f_\infty(y)} = -\frac12\log(1+\rho)+\frac1{2\sigma^2}\left(1-\frac1{1+\rho}\right)y^2.
\]
Define the CUSUM procedure
\begin{equation*}
\widetilde{T}= \inf\bigg\{ t : \max_{0\leq k<t}\sum_{i=k+1}^t \Big[\frac1{2\sigma^2}\left(1-\frac1{1+\rho}\right)(q^\intercal x_i)^2 -\frac{\log(1\!+\!\rho)}{2}\Big] \! \geq \! b'\bigg\},
\end{equation*}
where $b' = \frac1{2\sigma^2}(1-\frac1{1+\rho})(b-\frac{\sigma^2\log(1+\rho)}{1-1/(1+\rho)})w$, we then have
\[
\mathbb{E}_0[T_{\rm E}] \geq \mathbb{E}_0[\widetilde{T}].
\]
Since $\widetilde{T}$ is a CUSUM procedure with 
\[
\int\log\left[\frac{f_0(y)}{f_\infty(y)}\right]f_0(y)dy=-\frac12\log(1+\rho)+\frac{\rho}{2},
\]
by \cite{siegmund2013sequential}, we have:
\[
\mathbb{E}_0[\widetilde{T}]=\frac{e^{-b^\prime}+b^\prime-1}{-\log(1+\rho)/2+\rho/2}.
\]
This completes the proof.

\section*{Acknowledgements}
The work of Liyan Xie and Yao Xie was supported by US National Science Foundation (NSF) CCF-1442635, CMMI-1538746, DMS-1830210, and the Career Award CCF-1650913. The work of George V. Moustakides was supported by NSF CIF-1513373, through Rutgers University. The authors would like to thank the Editor and the anonymous referees for the thoughtful comments and suggestions, which led to an improvement of the presentation. 


\end{document}